\documentclass[12pt, twoside, leqno]{article}

\usepackage{amsmath,amsthm}
\usepackage{amssymb}

\usepackage{enumitem}

\usepackage{graphicx}

\usepackage[T1]{fontenc}

\newtheorem{theorem}{Theorem}[section]
\newtheorem{corollary}[theorem]{Corollary}
\newtheorem{lemma}[theorem]{Lemma}
\newtheorem{proposition}[theorem]{Proposition}

\theoremstyle{definition}

\numberwithin{equation}{section}

\usepackage{dsfont, amssymb, latexsym, graphicx, amsfonts, amsmath, mathtools, amsthm, geometry, environ, multicol, subfiles, titlesec, titletoc, fancyhdr, enumitem, xpatch, mathrsfs}

\newcommand\psum{\mathop{\sum\nolimits^{\mathrlap{*}}}}

\newcommand\lle{\mathop{\:\ll\:}_\varepsilon}

\newcommand\nn{\boldsymbol{\mathrm{N}}}

\newcommand\zz{\boldsymbol{\mathrm{Z}}}
\newcommand\qq{\boldsymbol{\mathrm{Q}}}
\newcommand\rr{\boldsymbol{\mathrm{R}}}

\newcommand\e{\boldsymbol{\mathrm{e}}}
\newcommand\dif{\mathrm{d}}

\begin{document}

	\title{A large sieve inequality for characters to quadratic moduli}
	

	\author{C. C. Corrigan\\
		School of Mathematics and Statistics\\ 
		University of New South Wales\\
		Sydney, Australia\\
		E-mail: c.corrigan@unsw.edu.au}
	
	\date{}
	
	\maketitle
	
	
	\renewcommand{\thefootnote}{}
	
	\footnote{2020 \emph{Mathematics Subject Classification}: Primary 11L40; Secondary 11B57; 11L07; 11L15; 11M26; 11N35.}
	
	\footnote{\emph{Key words and phrases}: Large sieve inequality, Farey fractions with special denominators, zero-density estimates.}
	
	\renewcommand{\thefootnote}{\arabic{footnote}}
	\setcounter{footnote}{0}
	
	
	\begin{abstract}
		In this article, we establish a large sieve inequality for additive characters to moduli in the range of appropriate integer polynomials of degree two.  As an application, we derive a weighted zero-density estimate for twists of $L$-functions associated to multiplicative characters with conductor of this form.
	\end{abstract}
	
	\section{Introduction.}
	The concept of the large sieve was first proposed by Linnik \cite{linnik,linnikog}, for the purpose of studying the distribution of quadratic non-residues.  To this work, developments were soon made by R\'enyi \cite{ren1, ren2, ren5, ren4, ren3, ren6, ren7, renyinew1,renyinew2,renyinew3} from a more probabilistic standpoint.  Later, Roth \cite{roth1, roth2} and Bombieri \cite{bom1} refined the work of R\'enyi to, essentially, the best possible form.  Davenport and Halberstam \cite{davhal} were the first to formulate the large sieve as an inequality of the type
	\begin{equation}\label{classic}
		\sum_{\alpha\in\mathscr{A}}\Big|\sum_{n\leqslant N}z_n\e(\alpha n)\Big|^2\ll\big(\delta^{-1}+N\big)\sum_{n\leqslant N}|z_n|^2,
	\end{equation}
	where $(z_n)_n$ is an arbitrary sequence of complex numbers, and $\mathscr{A}$ is a finite subset of the real line, such that $\|\alpha-\alpha'\|\geqslant\delta$ for all distinct $\alpha,\alpha'\in\mathscr{A}$ and some fixed $\delta>0$.  Here, as usual, for any real number $\xi$, we write $\e(\xi)=\exp(2\pi i\xi)$, and we let $\|\xi\|$ denote the distance from $\xi$ to the nearest integer.  Improvements to the implied constant in \eqref{classic} were made by Gallagher \cite{gallgs}, Bombieri and Davenport \cite{mont20}, Liu \cite{mont140}, and Bombieri \cite{bombieri}.  In particular, Selberg \cite{selberg44} demonstrated that
	\begin{equation*}
		\sum_{\alpha\in\mathscr{A}}\Big|\sum_{n\leqslant N}z_n\e(\alpha n)\Big|^2\leqslant\big(\delta^{-1}+N-1\big)\sum_{n\leqslant N}|z_n|^2,
	\end{equation*}
	which from the work of Bombieri and Davenport \cite{bomdav} we know to be sharp.  This had been previously shown by Montgomery and Vaughan \cite{monvau} with $N-1$ replaced by $N$.\par    
	Due to the arithmetical applications of the statement, particular attention has been given in the literature to the case where $\mathscr{A}$ is a collection of Farey fractions with denominators in a set of the form $f(\mathscr{Q})$, where $\mathscr{Q}\subset\nn$ is a finite interval, and $f:\nn\hookrightarrow\nn$ is strictly increasing.  For any such $f$, and any $Q,N\geqslant3$, we are therefore interested in the asymptotic size of the expression
	\begin{equation*}
		\mathscr{N}_f(Q,N)=\sup_{(z_n)_n\in\mathscr{C}_N}\sum_{q\leqslant Q}\sum_{\substack{a\leqslant f(q)\\(a,f(q))=1}}\Big|\sum_{n\leqslant N}z_n\e\Big(\frac{an}{f(q)}\Big)\Big|^2,
	\end{equation*}
	where we have denoted by $\mathscr{C}_N$ the set of sequences $(z_n)_n$ of complex numbers for which the $\ell^2$-norm of $(z_n)_{n\leqslant N}$ is exactly $1$.  Note that $\mathscr{N}_f(Q,N)$ is indeed well-defined, for by virtue of \eqref{classic}, we can establish the trivial bound
	\begin{equation}\label{trivial}
		\mathscr{N}_f(Q,N)\leqslant\min\big(Qf(Q)+QN,f(Q)^2+N\big).
	\end{equation}
	Examples of $f$, for which \eqref{trivial} is the optimal bound, were given by Elliott \cite{elliott71} and Montgomery \cite{mon78}, though Wolke \cite{wolke1,wolke2} demonstrated the existence of an $f$ such that \eqref{trivial} is not sharp for some ranges of $Q$ and $N$.  This work of Wolke pertains to prime moduli, and has recently been improved on by Iwaniec \cite{iwaniecsieve} to the best possible form, under some minor restrictions on the support of $(z_n)_n$.\par    
	In his doctoral thesis, Zhao \cite{acta,zhaothesis} demonstrated that, when $f$ is a monomial of degree two or greater, the estimate \eqref{trivial} is not sharp for certain choices of $Q$ and $N$.  Indeed, in this case, improvements to \eqref{trivial}, and the original work of Zhao, were made by Baier and Zhao \cite{baierzhaosq,baierzhaopm}, Halupczok \cite{halupczok}, Munsch \cite{munsch21}, and most recently by Baier and Lynch \cite{baierlynch}, Baker, Munsch, and Shparlinski \cite{BaMuSh}, and McGrath \cite{mcg}.  In the latter two articles, the case where $f$ is a general polynomial was also studied, inspired by earlier work of Halupczok \cite{hal15,hal18,halupczok2020} and Halupczok and Munsch \cite{halmun}.  In a more general setting, it was shown by Baier \cite{baier} that the majorisation
	\begin{equation}\label{baierbound}
		\mathscr{N}_f(Q,N)\lle N+f(Q)^{1+\varepsilon}\big(Q+\sqrt{N}\big)
	\end{equation}
	holds whenever $f(q\leqslant Q)$ is sufficiently well-distributed in the residue classes, which improves on \eqref{trivial} whenever $N\ll f(Q)^{1+\varepsilon}(Q+\sqrt{N})\ll f(Q)^2$.\par
	It is believed that \eqref{baierbound} is not sharp in the range $Q^2\ll N\ll f(Q)^{2+\varepsilon}$ whenever $f(q\leqslant Q)$ is well-distributed in the residue classes.  Indeed, in line with conjectures made in \cite{acta,elliott71,halupczok2020}, we expect a bound of the type
	\begin{equation}\label{conjecturalbound}
		\mathscr{N}_f(Q,N)\lle f(Q)^\varepsilon(Qf(Q)+N)
	\end{equation}
	to hold under the distribution condition of \eqref{baierbound}.  Now, appealing to the trivial bound \eqref{trivial}, we can show, for any increasing function $f$, that the conjectural bound \eqref{conjecturalbound} holds whenever $N\ll f(Q)$ or $N\gg f(Q)^2$.  In practice, however, it is most common to be interested in the case where $N\asymp Qf(Q)$, which clearly lies outside of both of these ranges.  Consequently, the aforementioned work on polynomial moduli has all been concentrated on extending the range of $N$ for which \eqref{conjecturalbound} holds.\par    
	Clearly, \eqref{conjecturalbound} holds for the identity map $f:q\mapsto q$, however, it is unknown if the conjecture holds for any monomial of degree two or greater.  The first steps towards \eqref{conjecturalbound} for the monomial $f:q\mapsto q^2$ were made by Zhao \cite{acta,zhaothesis} and Baier \cite{baier,baier2016}, and the current strongest unconditional result in this direction is due Baier and Zhao \cite{baierzhaosq}, who showed that
	\begin{equation}\label{baierzhaosq}
		\mathscr{N}_f(Q,N)\lle Q^{\varepsilon}\big(Q^3+N+\min\big(N\sqrt{Q},Q^2\sqrt{N}\big)\big),
	\end{equation}
	and thus demonstrated, in the case of square moduli, that the conjectural bound \eqref{conjecturalbound} also holds whenever $Q^2\ll N\ll Q^2\sqrt{Q}$.  In a recent preprint, under some conjectures related to higher additive-energies of modular square-roots, Baier \cite{bnew} has demonstrated a small power-saving on \eqref{baierzhaosq} around the central point $N\asymp Q^{3}$.\par    
	Now, the estimate \eqref{baierzhaosq}, and slight generalisations of it, have been used in recent literature for a number of applications.  For example, the large sieve for square moduli was used in \cite{bomvino,baker3} to study the distribution of primes in arithmetic progressions, in \cite{igorlee,bankspappalardishparlinski} to study elliptic curves over finite fields, in \cite{bourgainfordshparlinski} to study Fermat quotients, in \cite{bomvino,matomaki,palindromes} to study the representations of primes, and in \cite{lowly} to study one-level density estimates for elliptic curve $L$-functions.  In particular, with $s(m)$ denoting the unique square-free integer $a$ verifying $\qq(\sqrt{m})=\qq(\sqrt{a})$, it was shown by Baier and Zhao \cite{bomvino} (cf.\:\cite{matomaki,bai2}) that there are infinitely-many primes in the set
	\begin{equation*}
		\mathscr{T}_\vartheta=\big\{m:s(m-1)<m^\vartheta\big\},
	\end{equation*}
	provided that $\vartheta>\tfrac59$.  We note here that, using a Bombieri-Vinogradov type result of Baker \cite{baker3}, which itself makes use of \eqref{baierzhaosq}, it is possible to show that any $\vartheta>\tfrac12$ is admissible.  Note that, as $\vartheta\to\tfrac12$, the density of the set $\mathscr{T}_\vartheta$ approaches that of the set considered in the groundbreaking work of Friedlander and Iwaniec \cite{fi}, though the treatment of the sets $\mathscr{T}_\vartheta$ is significantly less involved, illustrating the strength of the result of Baker \cite{baker3}.  We remark that, under a certain Elliot-Halberstam type conjecture, the set $\mathscr{T}_\vartheta$ contains infinitely-many primes whenever $\vartheta>0$.\par    
	Our primary interest in this article, is to fully generalise \eqref{baierzhaosq} to arbitrary integer polynomials $f:\nn\hookrightarrow\nn$ of degree two, so that further applications of the large sieve inequality may be realised.  In this direction, we offer the following result.    
	\begin{theorem}\label{thm}
		Suppose that $f:\nn\hookrightarrow\nn$ is a strictly increasing polynomial of degree two, with leading coefficient $A$ and discriminant $\Delta_f$.  Then, for any $N\geqslant3$, if $Q\geqslant|\Delta_f|/A^2$ is sufficiently large that $3AQ^2\geqslant f(Q)$, then we have
		\begin{equation*}
			\mathscr{N}_f(Q,N)\lle f(Q)^\varepsilon\big(Qf(Q)+N+\min\big(N\sqrt{Q},f(Q)\sqrt{N}\big)\big),
		\end{equation*}
		where the implied constant depends on $\varepsilon$ alone.
	\end{theorem}
	The above improves the trivial bound \eqref{trivial} in the range $f(Q)\ll N\ll f(Q)^2$, and moreover confirms the conjectural bound \eqref{conjecturalbound} in the non-trivial range $f(Q)\ll N\ll f(Q)\sqrt{Q}$.  Now, we may follow a standard procedure, involving Gau\ss\:sums, to derive from Theorem~\ref{thm} an analogous large sieve inequality pertaining to multiplicative characters to moduli in the set $f(q\leqslant Q)$.  Results in this direction can be used to derive mean-value estimates for Dirichlet polynomials twisted by the aforementioned family of multiplicative characters, which in turn can be used to study the horizontal distribution of zeros of Dirichlet $L$-functions.  Such results, often referred to as zero-density estimates, concern the number
	\begin{equation*}
		N(\sigma,T,\chi)=\#\{\varrho\in R(\sigma,T):L(\varrho,\chi)=0\},
	\end{equation*}
	where $R(\sigma,T)=[\sigma,1]+i[-T,T]$ is a subregion of the critical strip $[0,1]+i\rr$, and $\chi$ is some Dirichlet character.  For any $Q\geqslant1$, Bombieri \cite{bom1,bombierifrench} and Montgomery \cite{montgomery,mon69} developed a procedure for bounding the average size of $N(\sigma,T,\chi)$ over all non-principal, or all primitive, Dirichlet characters $\chi$ of modulus not exceeding $Q$.  In \cite{cozh,msv,anote}, we obtained results pertaining to averages over thinner families of characters, and to these we add the following.
	\begin{corollary}\label{cor}
		Assume the hypothesis of Theorem~\ref{thm}, and suppose that $r$ is a natural number.  Then, for any $\sigma\in\:]\tfrac12,1[$ and $T\in\:]0,\infty[$, we have
		\begin{align*}
			&\sum_{\substack{q\leqslant Q\\(r,f(q))=1}}\sum_{\psi\bmod{r}}\psum_{\substack{\chi\bmod{f(q)}\\\psi\chi\neq\chi_0}}\frac{|\tau(\psi)|^2}{\varphi(r)}N(\sigma,T,\psi\chi)\\
			&\hspace{15mm}\lle(rf(Q)T)^\varepsilon\min\big(\sqrt{Q}(r\sqrt{Q}f(Q)T)^{3(1-\sigma)/(2-\sigma)},(r^3Qf(Q)^3T^2)^{(1-\sigma)/\sigma}\big),
		\end{align*}
		where $\tau(\psi)$ denotes the usual Gau\ss\:sum of $\psi$, and $\chi_0$ denotes the principal character of appropriate modulus.  The implied constant depends on $\varepsilon$ alone.
	\end{corollary}
	The above generalises a result from \cite{cozh} pertaining to square moduli.  None of the results in \cite{cozh}, however, feature an average over the characters $\psi$ modulo $r$, or rather they correpsond only to the case where $r=1$.  We note that any improvement to Theorem~\ref{thm} would lead to improvements to Corollary~\ref{cor}.\par    
	Now, in a different direction, we note that the trivial bound \eqref{trivial} gives us that the $f(Q)^\varepsilon$ factor present in the conjectural bound \eqref{conjecturalbound} can be removed if $N\ll f(Q)$ or $N\gg f(Q)^2$.  Indeed, for certain $f$, the factor $f(Q)^\varepsilon$ can be replaced, uniformly, by an absolute constant.  Nonetheless, the $f(Q)^\varepsilon$ factor was conjectured by Zhao \cite{acta} to be necessary in the case of square moduli, and, due to the work of Baier, Lynch, and Zhao \cite{BaLyZh}, it is known that in this case the $f(Q)^\varepsilon$ factor cannot even be replaced by a power of $\log f(Q)$.  Essentially, this is due to the fact that, when $Q$ has a large number of prime divisors, the number of Farey fractions of the form $a/q^2$ in short intervals greatly exceeds the expected number.  Interestingly, this argument follows over fairly easily to the case where $f:\nn\hookrightarrow\nn$ is an arbitrary integer polynomial of degree two.  More explicitly, we have the following result.
	\begin{theorem}\label{infthm}
		Suppose that $f:\nn\hookrightarrow\nn$ is as in Theorem~\ref{thm}.  Then, we have
		\begin{equation*}
			\liminf_{Q\to\infty}\frac{(\log f(Q))^J Qf(Q)}{\mathscr{N}_f(Q,Qf(Q))}=0
		\end{equation*}
		for any arbitrarily large $J>0$.
	\end{theorem}
	We now remark that, using an asymptotic result of Shapiro \cite{shapiro} pertaining to the mean value of $\varphi(f(q))$ over long intervals, we can easily show that the bound $\mathscr{N}_f(Q,N)\gg c_fQf(Q)+N$ holds under the hypothesis of Theorem~\ref{thm}, where $c_f$ is an explicit constant, depending only on $f$.  In fact, it can be shown by a trivial computation that $c_f\geqslant2-15/\pi^{2}$, thus giving us the expected lower bound
	\begin{equation*}
		\mathscr{N}_f(Q,N)\gg Qf(Q)+N.
	\end{equation*}
	This generalises an observation of Baier and Lynch \cite{baierlynch} from the case of square moduli to degree two integer polynomial moduli.\par
	{\textit{Remark on notation.}}  In this article, we will adopt mostly standard notations.  In particular, we let $\omega(n)$ and $\tau(n)$ denote the number of prime divisors, and positive divisors, of the natural number $n$, respectively.  Additionally, $\mu(n)$ and $\varphi(n)$ will be used to denote the M\"obius and Euler functions, respectively.  We adopt both the Landau and Vinogradov notations for asymptotics, and will write $q\sim Q$ to mean $Q<q\leqslant2Q$.  Also, $\varepsilon$ will denote an arbitrarily small, positive constant, and may vary in actual value even within the same line.  We lastly note that, unless otherwise obvious, all sums are taken over natural numbers, which we consider to be the strictly positive integers.\newline
	
	\section{The strategy.}
	In the literature, \eqref{classic} is often referred to as the classical large sieve inequality.  Davenport \cite{davenport} gave a slightly more general formulation of this inequality, which, in the weaker form of Baier \cite{baier}, can be stated as
	\begin{equation}\label{cls}
		\sum_{\alpha\in\mathscr{A}}\Big|\sum_{n\leqslant N}z_n\e(\alpha n)\Big|^2\ll K_\delta\big(\delta^{-1}+N\big)\sum_{n\leqslant N}|z_n|^2,
	\end{equation}
	where, adopting the nomenclature of Wolke \cite{wolke1,wolke2},
	\begin{equation}\label{clsafter}
		K_\delta=\sup_{\alpha\in\rr}P_\delta(\alpha)\quad\text{with}\quad P_\delta(\alpha)=\#\{\alpha'\in\mathscr{A}:|\alpha-\alpha'|\leqslant\delta\}.
	\end{equation}    
	Here, as in the previous section, $\mathscr{A}$ is an arbitrary finite collection of real numbers in the unit interval, and $(z_n)_n$ is an arbitrary non-zero sequence of complex numbers.  So, following a standard practice, we break the sum over the set $f(q\leqslant Q)$ in the definition of $\mathscr{N}_f(Q,N)$ into $O(\log f(Q))$ dyadic intervals $\mathscr{Q}$, in order to control the size of the moduli, and we aim to obtain a suitable estimate for $K_\delta$, where $\mathscr{A}$ is the set of Farey fractions with denominator in one such interval $\mathscr{Q}$.  More explicitly, we define these intervals as in the following result.
	\begin{lemma}\label{sizeofq}
		Under the hypothesis of Theorem~\ref{thm}, suppose that $M$ is a natural number, and take $\mathscr{Q}$ to be the set of natural numbers $q\leqslant Q$ for which $f(q)\sim AM^2$.  Then, $\mathscr{Q}$ is non-empty only when $M\ll Q$, in which case it is an interval satisfying $\#\mathscr{Q}\ll M$.  Moreover, we have $f(Q)\asymp AQ^2$.  The implied constants are all absolute.
	\end{lemma}
	\begin{proof}
		Let $B$ and $C$ denote the linear and constant coefficients of $f$, respectively.  Then, since $f$ and $f'$ are both strictly positive on the natural numbers, we see by considering specific values of $f$ and $f'$ that the inequalities
		\begin{equation*}
			A+B+C>0\quad\text{and}\quad2A+B>0
		\end{equation*}
		must hold.  We can use these facts to show, for any natural number $q$, that
		\begin{equation*}
			Aq^2+Bq+C>A(q^2-1)+B(q-1)>A(q^2-1)-2A(q-1)>A(q-1)^2
		\end{equation*}
		and additionally that
		\begin{equation*}
			-Bq-C<A-B(q-1)<A+2A(q-1)<3Aq
		\end{equation*}
		and thus, for any $q\geqslant3$ with $f(q)\sim AM^2$, we must have
		\begin{equation}\label{eqn:aqsq}
			Aq^2=f(q)\Big(1-\frac{Bq+C}{Aq^2+Bq+C}\Big)<2AM^2\Big(1+\frac{3q}{(q-1)^2}\Big)<9AM^2,
		\end{equation}
		so that $\#\mathscr{Q}\leqslant2+3M$.  Now, it is trivial that, when non-empty, $\mathscr{Q}$ must be an interval, for $f$ is monotonic.  Moreover, it is clear that if $M$ is such that $AM^2>f(Q)$, then $\mathscr{Q}$ must be empty, and thus $\mathscr{Q}$ is non-empty only if $M\ll\sqrt{f(Q)/A}$, so that $AQ^2\ll f(Q)$ by virtue of \eqref{eqn:aqsq}.  The proof is complete on noting that the condition $f(Q)\ll AQ^2$ is contained in the hypothesis.
	\end{proof}
	Following a standard procedure for estimating $P_\delta(\alpha)$, originally laid down by Wolke \cite{wolke1,wolke2}, we use a variant of the Dirichlet approximation theorem to write $\alpha=b/r+z$ for some appropriate Farey fraction $b/r$ and real number $z$ of small absolute value.  For reasons that will later become apparent, we require that the denominator $r$ be odd, and prime to $A$.  Now, as outlined in \cite{iwanieclecnotes,diophantine}, it is not difficult to show, using the sieve of Eratosthenes, that there exists a positive real number $t_A=O_\varepsilon(A^\varepsilon)$, depending only on $A$, such that the inequality
	\begin{equation}\label{mainlemma}
		\sup_{R\geqslant1}\min_{\substack{r\leqslant R\\(r,2A)=1}}R\|r\alpha\|\leqslant t_A
	\end{equation}
	holds for any irrational number $\alpha$.  This fact is the foundation of the following result, which is a crucial adaptation of the method of Baier and Zhao \cite{baierzhaosq}.
	\begin{lemma}\label{blemma}
		Assume that the hypothesis of Lemma~\ref{sizeofq} holds, and let $K_\delta$ and $P_\delta(\alpha)$ be defined as in \eqref{clsafter} with $\mathscr{A}$ being the set of Farey fractions with denominator in $\mathscr{Q}$.  For each natural number $r$, choose a prime $p_r=O_\varepsilon(r^\varepsilon)$, not dividing $r$, and with $t_A$ as in \eqref{mainlemma}, suppose that $t_{rA}\geqslant t_A$ is a prime, not dividing $r$, satisfying $t_{rA}=O_\varepsilon((rA)^\varepsilon)$.  Write $\kappa_{r\delta}=p_r/(r\sqrt{\delta})$, and, for any Farey fraction $b/r$, let $\mathscr{Z}_{b/r}$ denote the set of rational numbers $p_rt_{rA}/(kr^2)$, where $k\geqslant[\kappa_{r\delta}]+1$ is an integer satisfying $(A,p_rt_{rA}\overline{b}+rk)=1$, with $\overline{b}$ denoting a multiplicative inverse of $b$ modulo $r$.  Then, there exists a Farey fraction $b/r$ with $r\leqslant 1/\sqrt{\delta}$ prime to $2A$, and a real number $z\in\mathscr{Z}_{b/r}\cap[\delta t_{rA}/p_r,1]$, verifying the inequality $K_\delta\leqslant P_{\delta t_{rA}}(b/r+z)$.
	\end{lemma}
	\begin{proof}
		Firstly note that it is indeed possible, for any natural number $r$, to find a prime number $p_r=O_\varepsilon(r^\varepsilon)$ not dividing $r$.  For, the first $\omega(r)+1$ primes cannot all divide $r$, and thus the prime number theorem implies that the smallest prime not dividing $r$ must be $O_\varepsilon(r^\varepsilon)$.  The same argument allows us to establish the existence of a $t_{rA}$ satisfying the conditions of the assertion.\par
		Now, we note that to bound $K_\delta$, it suffices to bound $P_\delta(\alpha)$ for all $\alpha$ in the unit interval.  Moreover, by instead considering the number $P_{\delta(1+t_{rA}/p_r)}(\alpha)$, we may restrict our attention to the case where $\alpha$ is irrational.  So, fix an irrational $\alpha$ lying in the unit interval.  By \eqref{mainlemma}, there exists a Farey fraction $b/r$, with $r\leqslant1/\sqrt{\delta}$ prime to $2A$, such that $|b/r-\alpha|\leqslant t_{rA}\sqrt{\delta}/r$.  Now, since $P_{\delta(1+t_{rA}/p_r)}(\alpha)=P_{\delta(1+t_{rA}/p_r)}(1-\alpha)$, we may assume without loss of generality that $\alpha-b/r\geqslant0$.  It suffices to show for such $\alpha$ and $b/r$ that there exists a $z\in\mathscr{Z}_{b/r}\cap[\delta t_{rA}/p_r,1]$ with $P_{\delta(1+t_{rA}/p_r)}(\alpha)\leqslant P_{\delta t_{rA}}(b/r+z)$.\par
		Define the set $\mathscr{E}_{b/r}=\{p_rt_{rA}/(kr^2):k\geqslant[\kappa_{r\delta}]+1\}$, and note that any two elements of $\mathscr{E}_{b/r}$ have a spacing of at most $\delta t_{rA}/p_r$, and that there exists a $z\in\mathscr{E}_{b/r}$ such that $t_{rA}\sqrt{\delta}/r-z\leqslant \delta t_{rA}/p_r$.  Further, the maximal element of the intersection $\mathscr{E}_{b/r}\cap[\delta t_{rA}/p_r,1]$ is at least $p_r\delta t_{rA}/(p_r+1)$ and the minimal element is at most $2t_{rA}\delta/p_r$, so that the intersection $\mathscr{E}_{b/r}\cap[\delta t_{rA}/p_r,1]$ contains at least two elements.  These facts together imply that there exists a $z\in\mathscr{E}_{b/r}\cap[\delta t_{rA}/p_r,1]$ such that the inequality $P_{\delta(1+t_{rA}/p_r)}(\alpha)\leqslant P_{\delta(1+3t_{rA}/p_r)}(b/r+z)$ holds.  Now, since $\mathscr{E}_{b/r}\cap[\delta t_{rA}/p_r,1]$ contains at least two elements, and since $A$ cannot share a common factor with both $p_rt_{rA}\overline{b}+rk$ and $p_rt_{rA}\overline{b}+r(k+1)$, there must be a $z\in\mathscr{Z}_{b/r}\cap[\delta t_{rA}/p_r,1]$ verifying the inequality $P_{\delta(1+t_{rA}/p_r)}(\alpha)\leqslant P_{\delta(1+4t_{rA}/p_r)}(b/r+z)$, as required.
	\end{proof}
	The main advantage of the above result is that it allows us to treat Gau\ss\:sums involving the number $b/r+z$ more easily.  The necessity for the coprimality conditions $(r,p_rt_{rA})=1$ and $(A,p_rt_{rA}\overline{b}+rk)=1$ will become apparent in a later section, but are again in aid of simplifying the treatment of certain Gau\ss\:sums.  Now, adapting the work of Baier and Zhao \cite{baierzhaosq}, we can use Lemma~\ref{blemma} to derive the following result, which will serve as our principal estimate for the proof of Theorem~\ref{thm}.
	\begin{proposition}\label{prop}
		Let $\Delta_f$ be as in Theorem~\ref{thm}, and suppose that $M\geqslant1$ is sufficiently large that $|\Delta_f|\leqslant(AM)^2$.  Under the hypothesis of Lemma~\ref{blemma}, we have
		\begin{equation*}
			K_\delta\lle(AM)^\varepsilon\big(AM^3\delta+1+\min(\sqrt{M},AM^2\sqrt{\delta})\big),
		\end{equation*}
		where the implied constant depends on $\varepsilon$ alone.
	\end{proposition}
	Our main objective is now to establish the above result, for from this the main assertion will follow quickly.  Note that, in view of the trivial bound
	\begin{equation}\label{trivialbound111}
		K_\delta\ll\min(AM^3\delta+M,A^2M^4\delta+1),
	\end{equation}
	obtained as in \cite{acta}, we may restrict our attention to the case where $-\log\delta\asymp\log AM$.  Now, in the remainder of this article, as we have done above, we will make use of the standard bounds for $\omega(n)$, $\tau(n)$, and $\varphi(n)$ without mention, all of which can be found in the treatise of Hardy and Wright \cite{hawr} or that of Iwaniec and Kowalski \cite{iwakow}.\newline
	
	\section{Preliminary lemmata.}
	We now state, without proof, some preliminary results that are mostly already existing in the literature.  Our first preliminary is the Weyl estimate for exponential sums, a proof of which can be found in \S5.4 of \cite{tity}.
	\begin{lemma}\label{weylsum}
		Suppose that $N$ is a natural number, and that $\phi:\rr\to\rr$ is a polynomial of degree two, with leading coefficient $a$.  Then, we have the majorisation
		\begin{equation*}
			\Big|\sum_{n\leqslant N}\e(\phi(n))\Big|^2\ll N+\sum_{\ell<N}\min\big(N,\|2a\ell\|^{-1}\big),
		\end{equation*}
		where the implied constant is absolute. 
	\end{lemma}
	When applying the above result, we will require good estimates on the average size of $\|2a\ell\|$ for certain choices of $\phi$.  The following two results can be derived by adapting arguments appearing on p.163 and pp.170--171 of \cite{baierzhaosq}, respectively.
	\begin{lemma}\label{oldlemma43}
		Suppose that $b/r$ and $z$ are as in the assertion of Lemma~\ref{blemma}, and write $\beta=b/r+z$.  If $A$ is as in Lemma~\ref{sizeofq}, then, for any integers $n$ and $d$, we have 
		\begin{equation*}
			\#\{\ell\leqslant r:\|A\beta(nr+\ell)-d/r\|\leqslant1/(2r)\}\ll1,
		\end{equation*}
		where the implied constant is absolute.
	\end{lemma}
	\begin{lemma}\label{bigprelim}
		Let $N$ be a natural number, and suppose that $U,V,X\geqslant3$.  Moreover, suppose that $a\leqslant NX$ is a positive integer prime to $N$, and that $m,n\leqslant NX$ are positive integers prime to $a$.  Then, we have the majorisation
		\begin{equation*}
			\sum_{\substack{h\leqslant X\\(hN,a)=1}}\#\Big\{(u,v)\sim(U,V):\Big\|\frac{(m\overline{a}+n)v}{hN}\Big\|=\frac{u}{hN}\Big\}\lle\frac{UX}{m+an}+(NXUV)^\varepsilon\frac{UV}{N},
		\end{equation*}
		where $\overline{a}$ denotes any multiplicative inverse of $a$ modulo $hN$.  The implied constant depends on $\varepsilon$ alone.
	\end{lemma}
	We will on a number of occasions find ourselves needing to accurately evaluate Gau\ss\:sums involving certain polynomials of degree two, related in some way to $f$.  Consequently, we will make use of the following classical result, a proof of which can be found in the lecture notes of Graham and Kolesnik \cite{vdc} (cf.\:(3.4) of \cite{blomer}).
	\begin{lemma}\label{gausslemma}
		Suppose that $a$ and $b$ are natural numbers.  Then, for any odd natural number $h$, prime to $a$, we have the identity
		\begin{equation*}
			\sum_{m=1}^{h}\e\Big(\frac{am^2+bm}{h}\Big)=\e\Big(-\frac{\overline{4a}b^2}{h}\Big)\epsilon_{h}\chi_h(a)\sqrt{h},
		\end{equation*}
		where $\overline{4a}$ denotes a multiplicative inverse of $4a$ modulo $h$.  Here, $\epsilon_{h}$ lies on the unit circle, and depends only on the residue class of $h$ modulo $4$, and $\chi_h$ denotes the Jacobi symbol modulo $h$.
	\end{lemma}    
	Our next preliminary result is, essentially, due to Baier \cite{baier}, and is the foundation of part of our main argument.  It offers a convenient method for counting Farey fractions in short intervals, and can be proven along similar lines to Lemma~3 of \cite{baier}.
	\begin{lemma}\label{baierlemma}
		Suppose that the hypothesis of Lemma~\ref{blemma} holds, and that $y\geqslant1$ and $z\geqslant\delta t_{rA}/p_r$ are real numbers.  For a fixed Farey fraction $b/r$, write $\beta=b/r+z$ and
		\begin{equation*}
			\Pi_\delta(\beta,\eta,y)=\sum_{\substack{q\in\mathscr{Q}\\y-\eta\leqslant f(q)\leqslant y+\eta}}\sum_{\substack{(y-4p_r\eta)rz\leqslant m\leqslant (y+4p_r\eta)rz\\0\neq m\equiv-bf(q)\bmod{r}}}1,
		\end{equation*}
		where $\eta=\delta t_{rA}AM^2/(p_rz)$.  Then, we have the majorisation
		\begin{equation*}
			P_{\delta t_{rA}}(\beta)\ll1+\eta^{-1}\int\limits_{AM^2}^{2AM^2}\Pi_\delta(\beta,\eta,y)\:\dif y,
		\end{equation*}
		where the implied constant is absolute.
	\end{lemma}
	The above result, in one application, will lead us to consider certain exponential integrals.  These integrals will all be of the form presented in the following result, which, as in pp.164--166 of \cite{baierzhaosq}, can be derived using the method of stationary phase.
	\begin{lemma}\label{integrallemma}
		For any two integers $m$ and $\ell$, define
		\begin{equation*}
			I(m,\ell)=\int\limits_0^\infty e^{-y/AM^2}\e\Big(yzm-\frac{(r,m)\ell}{\sqrt{A}r}\sqrt{y}\Big)\:\dif y,
		\end{equation*}
		where $A,M,z,r$ are as in Lemma~\ref{blemma}.  The following three propositions hold.
		\begin{enumerate}[label={\upshape(\roman*)},itemsep=0mm]
			\item If $\ell=0$ and $m\neq0$, then $I(m,\ell)\ll1/|zm|$.
			\item If $\ell\neq0$ and $m/\ell\leqslant0$, then $I(m,\ell)\ll rAM/((r,m)\ell)$.
			\item If $\ell\neq0$ and $m/\ell>0$, then we have
			\begin{align*}
				I(m,\ell)&=\e(\tfrac18)\frac{(r,m)\ell}{zmr\sqrt{2Azm}}\e\Big(\frac{((r,m)\ell)^2}{4Azmr^2}\Big)\exp\Big(-\frac{((r,m)\ell)^2}{4(AMzmr)^2}\Big)\\
				&\hspace{70mm}+O\Big(\frac{1}{|zm|}+\frac{\sqrt{A}r}{(r,m)\ell\sqrt{zm}}\Big).
			\end{align*}
		\end{enumerate}
		The implied constants are all absolute.
	\end{lemma}
	Now, in the following two sections, we will derive three preliminary estimates for $P_{\delta t_{rA}}(\beta)$, which on comparison will be sufficient for the establishment of our principal estimate, Proposition~\ref{prop}.\newline    
	
	\section{Two preliminary estimates.}
	In this section, we give two estimates for the number $P_{\delta t_{rA}}(\beta)$, the proofs of which are not overly involved.  To derive our first estimate, we will follow the combinatorial method of Baier \cite{baier}, which uses Lemma~\ref{baierlemma} to reduce the problem to proving a result pertaining to the distribution of $\mathscr{Q}$ in the residue classes.  Firstly, however, we are required to prove the following result.
	\begin{lemma}\label{prehlemma}
		Suppose that $\kappa:\zz\to\zz$ is a polynomial of degree two, with leading coefficient $a$.  Then, for any odd prime power $p^\alpha$, we have
		\begin{equation*}
			\#\{\mathfrak{u}\in\zz/p^\alpha\zz:\kappa(\mathfrak{u})\subseteq p^\alpha\zz\}\leqslant(\alpha+1)\max\big((a,p^\alpha),\delta_{(a,p^\alpha)|\Delta_\kappa}(\Delta_\kappa/(a,p^\alpha),p^\alpha)\big),
		\end{equation*}
		where $\Delta_\kappa$ denotes the discriminant of $\kappa$.
	\end{lemma}
	\begin{proof}
		Appealing to the orthogonality identity for additive characters, we have
		\begin{equation*}
			\#\{\mathfrak{u}\in\zz/p^\alpha\zz:\kappa(\mathfrak{u})\subseteq p^\alpha\zz\}=p^{-\alpha}\sum_{m=1}^{p^\alpha}\sum_{n=1}^{p^\alpha}\e\Big(\frac{m\kappa(n)}{p^\alpha}\Big),
		\end{equation*}
		and therefore, writing $\kappa:n\mapsto an^2+bn+c$, it suffices to show that
		\begin{equation}\label{sufficestoshow}
			\Big|\sum_{m=1}^{p^\alpha}\sum_{n=1}^{p^\alpha}\e\Big(\frac{m\kappa(n)}{p^\alpha}\Big)\Big|\leqslant p^{\alpha}(\alpha+1)\max\big((a,p^\alpha),\delta_{(a,p^\alpha)|\Delta_\kappa}(\Delta_\kappa/(a,p^\alpha),p^\alpha)\big).
		\end{equation}
		So firstly, denoting by $\alpha'$ the unique non-negative integer satisfying $p^{\alpha'}=(a,p^\alpha)$, we can rewrite the double sum on the left-hand side of \eqref{sufficestoshow} as
		\begin{align}\label{transform}
			&\sum_{0\leqslant \gamma\leqslant\alpha}p^{\alpha-\gamma}\sum_{\substack{m=1\\(m,p)=1}}^{p^\gamma}\e\Big(\frac{mc}{p^\gamma}\Big)\sum_{n=1}^{p^\gamma}\e\Big(\frac{m(an^2+bn)}{p^\gamma}\Big)\notag\\
			&\hspace{12mm}=\sum_{0\leqslant \gamma\leqslant\alpha'}p^{\alpha-\gamma}\sum_{\substack{m=1\\(m,p)=1}}^{p^\gamma}\e\Big(\frac{mc}{p^\gamma}\Big)\sum_{n=1}^{p^\gamma}\e\Big(\frac{mbn}{p^\gamma}\Big)\notag\\
			&\hspace{30mm}+\delta_{p^{\alpha'}|b}\sum_{\alpha'<\gamma\leqslant\alpha}p^{\alpha'+\alpha-\gamma}\sum_{\substack{m=1\\(m,p)=1}}^{p^{\gamma}}\e\Big(\frac{mc}{p^\gamma}\Big)\sum_{n=1}^{p^{\gamma-\alpha'}}\e\Big(\frac{m(an^2+bn)}{p^\gamma}\Big),
		\end{align}
		the former expression of which we bound trivially by
		\begin{equation}\label{est1}
			\Big|\sum_{0\leqslant \gamma\leqslant\alpha'}p^{\alpha-\gamma}\sum_{\substack{m=1\\(m,p)=1}}^{p^\gamma}\e\Big(\frac{mc}{p^\gamma}\Big)\sum_{n=1}^{p^\gamma}\e\Big(\frac{mbn}{p^\gamma}\Big)\Big|\leqslant (1+\alpha')p^{\alpha+\alpha'}.
		\end{equation}
		Now, in the case where $p^{\alpha'}$ does not divide $b$, \eqref{sufficestoshow} follows immediately from \eqref{transform} and \eqref{est1}.  Thus, we assume for the remainder of the demonstration that $p^{\alpha'}$ divides $b$, and thus also $\Delta_\kappa$, and we shall put $a'=ap^{-\alpha'}$ and $b'=bp^{-\alpha'}$.  Moreover, we may assume that $\alpha'<\alpha$, so that $a'$ is prime to $p$, for otherwise the second expression on the right-hand side of \eqref{transform} would be an empty sum.  Under these assumptions, we may use Lemma~\ref{gausslemma} to rewrite the aforementioned expression as
		\begin{align}\label{nexttransform}
			&\sum_{\alpha'<\gamma\leqslant\alpha}p^{\alpha'+\alpha-\gamma}\sum_{\substack{m=1\\(m,p)=1}}^{p^{\gamma}}\e\Big(\frac{mc}{p^\gamma}\Big)\sum_{n=1}^{p^{\gamma-\alpha'}}\e\Big(\frac{m(a'n^2+b'n)}{p^{\gamma-\alpha'}}\Big)\notag\\
			&\hspace{27mm}=p^\alpha\sum_{\alpha'<\gamma\leqslant\alpha}\frac{\epsilon_{p^{\gamma-\alpha'}}}{\sqrt{p^{\gamma-\alpha'}}}\sum_{\substack{m=1\\(m,p)=1}}^{p^{\gamma}}\chi_{p^{\gamma-\alpha'}}(ma')\e\Big(\frac{m(c-\overline{4a'}(b')^2p^{\alpha'})}{p^\gamma}\Big)\notag\\
			&\hspace{27mm}=p^\alpha\sum_{\alpha'<\gamma\leqslant\alpha}\frac{\epsilon_{p^{\gamma-\alpha'}}}{\sqrt{p^{\gamma-\alpha'}}}\sum_{\substack{m=1\\(m,p)=1}}^{p^{\gamma}}\chi_{p}(ma')^{\gamma-\alpha'}\e\Big(-\frac{m\overline{4a'}\Delta_\kappa/p^{\alpha'}}{p^{\gamma}}\Big),
		\end{align}
		where $\overline{4a'}$ denotes any multiplicative inverse of $4a'$ modulo $p^{\alpha}$.  So, denote by $\alpha''$ the unique non-negative integer satisfying $p^{\alpha''}=(\Delta_\kappa/p^{\alpha'},p^\alpha)$, let $\xi=\max(\alpha',\alpha'')$, and split the sum over $\gamma$ on the right-hand side of \eqref{nexttransform} into two subsums, according to whether $\gamma\leqslant\xi$ or not.  The magnitude of the contribution of the former subsum to the right-hand side of \eqref{nexttransform} is precisely
		\begin{align}\label{e:1}
			\Big|p^\alpha\sum_{\alpha'<\gamma\leqslant\xi}\frac{\epsilon_{p^{\gamma-\alpha'}}}{\sqrt{p^{\gamma-\alpha'}}}\sum_{\substack{m=1\\(m,p)=1}}^{p^{\gamma}}\chi_p(ma')^{\gamma-\alpha'}\Big|&\leqslant\sum_{\alpha'<\gamma\leqslant\xi}\frac{p^\alpha\varphi(p^{\gamma})}{\sqrt{p^{\gamma-\alpha'}}}\leqslant(\xi-\alpha')p^{\alpha+\xi}.
		\end{align}
		Before treating the latter subsum, firstly note we may assume that $\Delta_\kappa/p^{\alpha'+\alpha''}$ is prime to $p$, for otherwise the range of summation would be empty.  This observation allows us to rewrite
		\begin{align*}    
			&\sum_{\xi<\gamma\leqslant\alpha}\frac{\epsilon_{p^{\gamma-\alpha'}}}{\sqrt{p^{\gamma-\alpha'}}}\sum_{\substack{m=1\\(m,p)=1}}^{p^{\gamma}}\chi_p(ma')^{\gamma-\alpha'}\e\Big(-\frac{m\overline{4a'}\Delta_\kappa/p^{\alpha'}}{p^{\gamma}}\Big)\\
			&\hspace{25mm}=\sum_{j\in\{0,1\}}\sum_{\substack{\xi<\gamma\leqslant\alpha\\\gamma-\alpha'\,\equiv\,j\bmod{2}}}\frac{\epsilon_{p^{j}}\chi_p(\Delta_\kappa/p^{\alpha'+\alpha''})^{j}}{\sqrt{p^{\gamma-\alpha'}}}p^{\alpha''}\sum_{\substack{m=1\\(m,p)=1}}^{p^{\gamma-\alpha''}}\chi_p(m)^{j}\e\Big(\frac{m}{p^{\gamma-\alpha''}}\Big),
		\end{align*}
		and so, appealing to the standard bound for Gau\ss\:sums, we see by \eqref{nexttransform} and \eqref{e:1} that the magnitude of the second term on the right-hand side of \eqref{transform} does not exceed
		\begin{equation*}
			(\xi-\alpha')p^{\alpha+\xi}+\sum_{\xi<\gamma\leqslant\alpha}\frac{p^{\alpha+\alpha''}\sqrt{p^{\gamma-\alpha''}}}{\sqrt{p^{\gamma-\alpha'}}}\leqslant(\alpha-\alpha')p^{\alpha+\xi}.
		\end{equation*}
		When combined with \eqref{transform} and \eqref{est1}, this is sufficient to establish the estimate \eqref{sufficestoshow} in the remaining case, and thus the proof is complete.
	\end{proof}
	Now, it is clear that, with rather little extra effort, the above result could be improved substantially.  Indeed, if the linear coefficient of the polynomial $\kappa$ is zero, then Lemma~\ref{prehlemma} could be improved by simply appealing to a result of Soundararajan \cite{sound}.  Any improvements, however, would not increase the effectiveness of its application in deriving the following key result.
	\begin{lemma}\label{hdlemma}
		Suppose that the hypothesis of Lemma~\ref{baierlemma} holds, and denote by $B$ and $C$ the two non-leading coefficients of $f$.  With $\Delta_f$ as in Theorem~\ref{thm}, suppose that $M\geqslant3$ is sufficiently large that $|\Delta_f|\leqslant(AM)^2$.  For any odd natural number $d$, prime to $A$, write $\mathscr{H}_d=\{h\in\nn:hd+C\in f(\mathscr{Q})\}$, and let $Y_d(x,\varrho)$ denote the closed ball, centred at $x/d$, of radius $\varrho/d$.  Then, if $v$ is an odd natural number, prime to $A$, and $\mathfrak{v}\in\zz/v\zz$ is such that $(4Adu+B^2,v)=n$ for any $u\in\mathfrak{v}$, we have
		\begin{equation*}
			\sup_{y\sim AM^2}\#\big(\mathscr{H}_d\cap\mathfrak{v}\cap Y_d(y-C,\eta)\big)\lle (Adv)^\varepsilon\frac{n}{(d,n)}\Big(d+\frac{\eta}{vAM}\Big),
		\end{equation*}
		where the implied constant depends on $\varepsilon$ alone.
	\end{lemma}
	\begin{proof}
		Suppose that $\ell|d$, and write $\mathscr{Q}_\ell=\{q\in\mathscr{Q}:(q,d)=\ell\}$.  With $\gamma_\ell=(A\ell,B)$, we write $A_\ell=A\ell/\gamma_\ell$ and $B_\ell=B/\gamma_\ell$, as well as $d_\ell=d/(\gamma_\ell\ell,d)$.  Now observe that, if $q\ell\in\mathscr{Q}_\ell$, then $d|(Aq\ell+B)q\ell$ precisely when $d_\ell|(A_\ell q+B_\ell)$, and thus $q$ must lie in a unique primitive residue class modulo $d_\ell$, say $\mathfrak{d}_\ell$.  So, letting $\mathscr{V}_\ell$ denote the set of residue classes $\mathfrak{w}\in\zz/d_\ell v\zz$ for which $\{(Aq\ell+B)q\ell/d:q\in\mathfrak{w}\}\subseteq\mathfrak{v}$, we have
		\begin{equation}\label{firstintersection}
			\mathscr{H}_d\cap\mathfrak{v}\subseteq\bigcup_{\ell|d}\bigcup_{\mathfrak{w}\in\mathscr{V}_\ell}\mathscr{D}_\ell(\mathfrak{w}),\quad\text{where}\quad\mathscr{D}_\ell(\mathfrak{w})=\{(Aq\ell+B)q\ell/d:q\in\mathfrak{d}_\ell\cap\mathfrak{w}\}.
		\end{equation}       
		Note that, by the Chinese remainder theorem, if $\mathfrak{w}\in\zz/d_\ell v\zz$ is such that $\mathfrak{d}_\ell\cap\mathfrak{w}\neq\varnothing$, then we must have $\mathfrak{d}_\ell\cap\mathfrak{w}=\mathfrak{w}$, as $d_\ell$ clearly divides $d_\ell v$.\par
		Now we consider the intersections $\mathscr{D}_\ell(\mathfrak{w})\cap Y_d(y-C,\eta)$, and so suppose that $q\in\mathfrak{w}$ satisfies $(Aq\ell+B)q\ell\in Y_1(y-C,\eta)$.  Indeed, using the quadratic formula and Taylor approximation, we deduce that there exists an absolute constant $\xi>0$ such that
		\begin{equation*}
			\mathscr{D}_\ell(\mathfrak{w})\cap Y_d(y-C,\eta)\hookrightarrow\big\{q\in\mathfrak{w}:\big|2A\ell q+B-\sqrt{\Delta_f+4Ay}\big|\leqslant\xi\eta/M\big\},
		\end{equation*}
		provided that $|\Delta_f|\leqslant(AM)^2$.  This, by virtue of \eqref{firstintersection}, yields the majorisation
		\begin{equation}\label{keypartoflemma}
			\#\big(\mathscr{H}_d\cap\mathfrak{v}\cap Y_d(y-C,\eta)\big)\ll\sum_{d|\ell}\sum_{\mathfrak{w}\in\mathscr{V}_\ell}\Big(1+\frac{\eta}{d_\ell v\ell AM}\Big),
		\end{equation}
		which leads us to consider the size of $\mathscr{V}_\ell$.\par
		Now, by fixing a $u\in\mathfrak{v}$ and writing $\kappa(q)=(Aq\ell+B)q\ell d_\ell/d-d_\ell u$, we see that the set $\mathscr{V}_\ell$ is precisely the set of $\mathfrak{w}\in\zz/d_\ell v\zz$ satisfying $\kappa(\mathfrak{w})\subseteq d_\ell v\zz$.  Hence, following a standard argument involving the Chinese remainder theorem, we can show that
		\begin{align}\label{hashcee}
			\#\mathscr{V}_\ell&=\prod_{p^\alpha\|d_\ell v}\#\{\mathfrak{w}\in\zz/p^\alpha\zz:\kappa(\mathfrak{w})\subseteq p^\alpha\zz\}\notag\\
			&\leqslant\tau(d_\ell v)\prod_{p^\alpha\| d_\ell v}\max\big((A\ell^2d_\ell/d,p^\alpha),\delta_{(A\ell^2d_\ell/d,p^\alpha)|\Delta_\kappa}(\Delta_\kappa/(A\ell^2d_\ell/d,p^\alpha),p^\alpha)\big),
		\end{align}
		by virtue of Lemma~\ref{prehlemma}.  Firstly note that, since $A$ is prime to both $\ell d_\ell$ and $v$, we must have $(A\ell^2d_\ell/d,d_\ell v)\leqslant\ell$.  Moreover, if $(A\ell^2d_\ell/d,p^\alpha)|\Delta_\kappa$ for a certain $p^\alpha\|d_\ell v$, then we see that $(\Delta_\kappa/(A\ell^2d_\ell/d,p^\alpha),p^\alpha)\leqslant (nd_\ell/(d,n),p^\alpha)$.  Hence, in view of \eqref{hashcee}, we must have $\#\mathscr{V}_\ell\leqslant\tau(d_\ell v)nd_\ell\ell/(d,n)$, from which the assertion follows immediately, on appealing to \eqref{keypartoflemma}.
	\end{proof}
	The difficulty in deriving our first estimate for $P_{\delta t_{rA}}(\beta)$ is contained in the above proof.  Indeed, it is a simple task to derive from Lemmata~\ref{baierlemma}~and~\ref{hdlemma} the following.    
	\begin{lemma}\label{prop1}
		Suppose that the hypothesis of Proposition~\ref{prop} holds, and moreover that $(AM^2)^{-2}\leqslant\delta\leqslant(AM^2)^{-1}$.  Then, with $\beta=b/r+z$ as in Lemma~\ref{baierlemma}, we have
		\begin{equation*}
			P_{\delta t_{rA}}(\beta)\lle1+(AM)^\varepsilon(AM^3\delta+AM^2rz),
		\end{equation*}
		where the implied constant depends on $\varepsilon$ alone.
	\end{lemma}
	\begin{proof}
		Let $\Pi_\delta(\beta,\eta,y)$ be defined as in Lemma~\ref{baierlemma}, and $\mathscr{H}_d$ and $Y_d(x,\varrho)$ be defined as in Lemma~\ref{hdlemma}.  Letting $\overline{b}$ denote a multiplicative inverse of $b$ modulo $r$, we have
		\begin{align}\label{sumtrans}
			\Pi_\delta(\beta,\eta,y)&=\sum_{d|r}\:\sum_{\substack{h\in\mathscr{H}_d\cap Y_d(y-C,\eta)\\(h,r/d)=1}}\:\:\sum_{\substack{m\in Y_d(bC+yrz,4p_r\eta rz)\backslash\{bC/d\}\\m\equiv-bh\bmod{r/d}}}1\notag\\
			&=\sum_{d|r}\:\sum_{\substack{m\in Y_d(bC+yrz,4p_r\eta rz)\backslash\{bC/d\}\\(m,r/d)=1}}\:\:\sum_{\substack{h\in\mathscr{H}_d\cap Y_d(y,\eta)\\h\equiv-\bar{b}m\bmod{r/d}}}1\notag\\
			&=\sum_{d|r}\sum_{n|r/d}\:\:\sum_{\substack{m\in Y_d(bC+yrz,4p_r\eta rz)\backslash\{bC/d\}\\(4Adm+bB^2,r/d)=n\\(m,r/d)=1}}\:\:\sum_{\substack{h\in\mathscr{H}_d\cap Y_d(y,\eta)\\h\equiv-\bar{b}m\bmod{r/d}}}1,
		\end{align}
		which allows us to apply Lemma~\ref{hdlemma}.  Firstly, however, we note that if $m$ is such that $(4Adm+bB^2,r/d)=n$, then we must have $(d,n)|bB^2$, so that $m$ must belong to a certain residue class modulo $n/(d,n)$, say $\mathfrak{m}_{dn}$.  Hence, appealing to \eqref{sumtrans} and recalling that $\eta\leqslant AM^2$, we have by virtue of Lemmata~\ref{baierlemma}~and~\ref{hdlemma} the majorisation
		\begin{align*}
			P_{\delta t_{rA}}(\beta)&\lle1+\frac{(AM)^\varepsilon}{\eta}\sum_{d|r}\sum_{n|r/d}\frac{nd}{(d,n)}\Big(1+\frac{\eta}{rAM}\Big)\int\limits_{AM^2}^{2AM^2}\sum_{\substack{m\in Y_d(bC+yrz,4p_r\eta rz)\backslash\{bC/d\}\\m\in\mathfrak{m}_{dn}}}1\:\dif y\\
			&\lle1+(AM)^\varepsilon\sum_{d|r}\sum_{n|r/d}\big(AM^2zr+M\eta z\big),
		\end{align*}        
		from which the assertion follows immediately.
	\end{proof}
	The above result is clearly sufficient to derive Proposition~\ref{prop} with the second term in the minimum.  Now, in order to derive our second preliminary estimate, we will adapt an analytic argument of Baier and Zhao \cite{baierzhaosq}.  So, firstly, following the work of Zhao \cite{acta}, we introduce the function $\phi:\rr\to\rr$, defined by
	\begin{equation}\label{phidef}
		\phi:x\mapsto\begin{cases}\psi(x)&\text{if }x\neq0\\\tfrac14\pi^2&\text{if }x=0\end{cases}\quad\text{where}\quad\psi:x\mapsto\Big(\frac{\sin\pi x}{2x}\Big)^2.
	\end{equation}
	Clearly $\phi(x)\geqslant0$ for any real $x$, and moreover $\phi(x)\geqslant1$ whenever $|2x|\leqslant1$.  We can also show that the Fourier transform of $\phi$ is given by $\hat{\phi}:x\mapsto\tfrac14\pi^2\max(0,1-|x|)$.  In the remainder of this article, we will use these facts without mention.  We are now ready to prove our second preliminary estimate for $P_{\delta t_{rA}}(\beta)$, which is as follows.
	\begin{lemma}\label{prop2}
		Under the hypothesis of Lemma~\ref{prop1}, we have
		\begin{equation*}
			P_{\delta t_{rA}}(\beta)\lle(AM)^\varepsilon\big(AM^3\delta+M/\sqrt{r}+\sqrt{M}\big),
		\end{equation*}
		where the implied constant depends on $\varepsilon$ alone.
	\end{lemma}
	\begin{proof}
		With $\phi$ as in \eqref{phidef}, we clearly have
		\begin{equation*}
			P_{\delta t_{rA}}(\beta)\leqslant\sum_{q\in\mathscr{Q}}\sum_{a\in\zz}\phi\Big(\frac{a-\beta f(q)}{4AM^2\delta t_{rA}}\Big),
		\end{equation*}
		where the inner sum on the right-side converges by virtue of the Weierstra\ss\:$M$-test.  Application of the Poisson summation formula and Cauchy's inequality then gives
		\begin{align}\label{old43}
			P_{\delta t_{rA}}(\beta)&\ll AM^2\delta t_{rA}\sum_{|h|<1/(4AM^2\delta t_{rA})}\sum_{q\in\mathscr{Q}}\big(1-4AM^2\delta t_{rA}|h|\big)\e(\beta hf(q))\notag\\
			&\ll AM^3\delta t_{rA}+AM^2\delta t_{rA}\sum_{h<1/(4AM^2\delta t_{rA})}\Big|\sum_{q\in\mathscr{Q}}\e(\beta hf(q))\Big|\notag\\
			&\ll AM^3\delta t_{rA}+M\sqrt{A\delta t_{rA}}\Big(\sum_{h<1/(4AM^2\delta t_{rA})}\Big|\sum_{q\in\mathscr{Q}}\e(\beta hf(q))\Big|^2\Big)^\frac{1}{2}.
		\end{align}
		Now, by Lemma~\ref{sizeofq}, $\mathscr{Q}$ is an interval of length $O(M)$.  Hence, by Lemma~\ref{weylsum}, there exists an absolute constant $\xi\geqslant 1$ such that
		\begin{equation*}
			\Big|\sum_{q\in\mathscr{Q}}\e(\beta hf(q))\Big|^2\ll M+\sum_{m<\xi M}\min\big(M,\|2A\beta hm\|^{-1}\big),
		\end{equation*}
		and so, we see that the second term on the right-hand side of \eqref{old43} is majorised by
		\begin{equation}\label{nodiggity}
			M\sqrt{A\delta t_{rA}}\Big(\frac{1}{AM\delta t_{rA}}+\sum_{m<\xi/(2AM\delta t_{rA})}\tau(m)\min\big(M,\|A\beta m\|^{-1}\big)\Big)^\frac{1}{2}.
		\end{equation}
		Splitting the sum over $m$ in \eqref{nodiggity} into $r$ subsums over the residue classes modulo $r$, and applying Lemma~\ref{oldlemma43}, we have the estimate
		\begin{align*}
			\sum_{m<\xi/(2AM\delta t_{rA})}\min\big(M,\|A\beta m\|^{-1}\big)&\ll\sum_{n<\xi/(2rAM\delta t_{rA})}\sum_{\ell\leqslant r}\min\big(M,\|A\beta(r(n-1)+\ell)\|^{-1}\big)\\
			&\ll\sum_{n<\xi/(2rAM\delta t_{rA})}\Big(M+\sum_{d<r}\Big(\frac{d}{r}-\frac{1}{2r}\Big)^{-1}\Big)\\
			&\ll\frac{M+r\log r}{rAM\delta t_{rA}},
		\end{align*}
		from which the assertion follows on appealing to \eqref{old43} and \eqref{nodiggity}, and the fact that $\log r=O_\varepsilon((AM)^\varepsilon)$ under the hypothesis of the assertion.
	\end{proof}
	If $r\gg M$, then the above result can be used to establish Proposition~\ref{prop} with the first term in the minimum.  If, however, this is not the case, then we appeal to a third preliminary result, which we shall derive in the first part of the following section.\newline
	
	\section{The principal estimate.}
	Our main tool in this section will be the following result, which is, essentially, an alternate formulation of Lemma~\ref{baierlemma}.
	\begin{lemma}\label{keylemma}
		Assume the hypothesis of Lemma~\ref{hdlemma}, and write $r^*=r/(r,m)$ and $m^*=m/(r,m)$.  Suppose that $\phi$ is as in \eqref{phidef}, that $I(m,\ell)$ is as in Lemma~\ref{integrallemma}, and, for an appropriate absolute constant $\xi>0$, define
		\begin{equation*}
			S(m,\ell)=\frac{\hat{\phi}(\xi p_r\eta zm)}{r^*}\hat{\phi}\Big(\frac{\xi\eta\ell}{r^*AM}\Big)\e\Big(\frac{B\ell}{2Ar^*}-\frac{\Delta_fzm}{4A}\Big)\sum_{d=1}^{r^*}\e\Big(\frac{m^*bf(d)+\ell d}{r^*}\Big)I(m,\ell)
		\end{equation*}
		Then, we have the majorisation
		\begin{equation*}
			P_{\delta t_{rA}}(\beta)\ll1+\frac{p_r\eta z}{AM}\Big|\sum_{m\in\zz}\sum_{\ell\in\zz}S(m,\ell)\Big|,
		\end{equation*}
		where the implied constant is absolute.
	\end{lemma}
	\begin{proof}
		If $y\geqslant AM^2$ is real, then we must have $y+\Delta_f/(4A)\gg AM^2$.  Hence, on noting that $(2Aq+B)^2=4Af(q)+\Delta_f$, we may use Taylor approximation to show that, for any natural number $q$, satisfying $y-\eta\leqslant f(q)\leqslant y+\eta$, we have
		\begin{equation*}
			|q-\lambda(y)|\leqslant\frac{\xi\eta}{2AM}\quad\text{where}\quad\lambda(y)=\sqrt{\frac{y+\Delta_f/(4A)}{A}}-\frac{B}{2A},
		\end{equation*}
		provided that $\xi\geqslant1$ is sufficiently large.  Consequently, with $\Pi_\delta(\beta,\eta,y)$ defined as in Lemma~\ref{baierlemma}, for an appropriate $\xi\geqslant8$, we have the majorisation
		\begin{equation}\label{topeq}
			\Pi_\delta(\beta,\eta,y)\ll\sum_{q\in\zz}\sum_{m\in\zz}\phi\Big(\frac{q-\lambda(y)}{\xi\eta}AM\Big)\phi\Big(\frac{mr-bf(q)-yrz}{\xi p_r\eta rz}\Big).
		\end{equation}
		Splitting the sum over $q\in\zz$ on the right-hand side of \eqref{topeq} into subsums over the residue classes modulo $r^*$, and twice applying the Poisson summation formula to the right-hand side of \eqref{topeq} yields
		\begin{equation}\label{boteq}
			\eta^{-1}\Pi_\delta(\beta,\eta,y)\ll\frac{p_r\eta z}{AM}H(y),
		\end{equation}
		where we have written
		\begin{equation*}
			H(y)=\sum_{m\in\zz}\sum_{\ell\in\zz}\frac{\hat{\phi}(\xi p_r\eta zm)}{r^*}\hat{\phi}\Big(\frac{\xi\eta\ell}{r^*AM}\Big)\e\Big(yzm-\frac{\ell\lambda(y)}{r^*}\Big)\sum_{d=1}^{r^*}\e\Big(\frac{m^*bf(d)+\ell d}{r^*}\Big).
		\end{equation*}
		On noting that $y\asymp AM^2$ whenever $AM^2+\Delta_f/(4A)\leqslant y\leqslant2AM^2+\Delta_f/(4A)$, we deduce from \eqref{boteq} the majorisation
		\begin{align*}
			\eta^{-1}\int\limits_{AM^2}^{2AM^2}\Pi_\delta(\beta,\eta,y)\:\dif y&\ll\frac{p_r\eta z}{AM}\int\limits_{AM^2+\Delta_f/(4A)}^{2AM^2+\Delta_f/(4A)}H(y-\Delta_f/(4A))\:\dif y\\
			&\ll\frac{p_r\eta z}{AM}\int\limits_{0}^{\infty}e^{-y/(AM^2)}H(y-\Delta_f/(4A))\:\dif y,
		\end{align*}
		which is sufficient to establish the assertion, by virtue of Lemma~\ref{baierlemma}.
	\end{proof}
	As we wish to apply Lemma~\ref{integrallemma} to estimate the integral $I(m,\ell)$ in the definition of $S(m,\ell)$, we break the sum over $(m,\ell)\in\zz\times\zz$ into two subsums, according to whether $m\ell\leqslant0$ or not.  Considering firstly the former, we have the following.
	\begin{lemma}\label{old63}
		Under the hypothesis of Lemma~\ref{keylemma}, we have
		\begin{equation*}
			\frac{\eta z}{AM}\mathop{\sum_{m\in\zz}\sum_{\ell\in\zz}}_{m\ell\leqslant0}S(m,\ell)\lle(AM)^\varepsilon\big(AM^3\delta+\sqrt{r}+M\delta/(z\sqrt{r})\big),
		\end{equation*}
		where the implied constant depends on $\varepsilon$ alone.
	\end{lemma}
	\begin{proof}
		Firstly note that, when $m=0$, we have $r^*=1$, and thus we have
		\begin{equation*}
			S(0,0)=\tfrac{1}{16}\pi^4\int\limits_0^\infty e^{-y/AM^2}\:\dif y\ll AM^2.
		\end{equation*}
		Moreover, by Lemma~\ref{integrallemma} and the standard bound for Gau\ss\:sums, we have
		\begin{equation*}
			\sum_{m\in\zz\backslash\{0\}}S(m,0)\ll\sum_{m<1/(\xi p_r\eta z)}\frac{1}{zm\sqrt{r^*}}=\frac{1}{z\sqrt{r}}\sum_{d|r}\sqrt{d}\sum_{\substack{m<1/(\xi p_r\eta z)\\(r,m)=d}}\frac{1}{m}\lle(AM)^\varepsilon\frac{1}{z\sqrt{r}}.
		\end{equation*}
		Similarly, appealing additionally to the symmetry of $\hat{\phi}$, we derive the bound
		\begin{equation*}
			\mathop{\sum_{m\in\zz}\sum_{\ell\in\zz\backslash\{0\}}}_{m\ell\leqslant0}S(m,\ell)\ll\sum_{0\leqslant m<1/(\xi p_r\eta z)}\sum_{\ell<r^*AM/(\xi\eta)}\frac{AM\sqrt{r^*}}{\ell}\lle(AM)^{1+\varepsilon}\Big(1+\frac{\sqrt{r}}{\eta z}\Big),
		\end{equation*}
		where we have again used the fact that $r^*=1$ when $m=0$.  The assertion follows on combining the above three deductions.
	\end{proof}
	Our treatment of the latter is slightly more involved.  We will use the following two results to bound the sum over $(m,\ell)\in\nn\times\nn$, from which we can state a bound for the sum of the remaining terms as a porism, by the symmetry of $\hat{\phi}$.
	\begin{lemma}\label{prerlemma}
		Assume the hypothesis of Lemma~\ref{keylemma} holds, write $k^*=p_rt_{rA}/(zr^*r)$, and let $\overline{r^*}$ denote the multiplicative inverse of $r^*$ modulo $p_rt_{rA}Am^*$.  Then, we have
		\begin{align*}
			&\sum_{m<1/(\xi p_r\eta z)}\frac{1}{\sqrt{mr^*}}\Big(\sum_{\ell\leqslant(AM)^{1+\varepsilon}mzr^*}\min\Big(AMmzr^*,\Big\|\frac{2\ell(\overline{b}p_rt_{rA}\overline{r^*}+k^*)}{p_rt_{rA}Am^*}\Big\|^{-1}\Big)\Big)^\frac{1}{2}\\
			&\hspace{90mm}\lle(AM)^\varepsilon\frac{\sqrt{AM}+M\sqrt{Azr}}{\eta\sqrt{z}},
		\end{align*}
		where the implied constant is absolute.
	\end{lemma}
	\begin{proof}
		We firstly rewrite the quantity of interest as
		\begin{equation}\label{ggsum}
			\sum_{d|r}\sum_{\substack{h<1/(\xi p_r\eta zd)\\(h,r/d)=1}}\frac{1}{\sqrt{rh}}\Big(\sum_{\ell\leqslant(AM)^{1+\varepsilon}hzr}\min\Big(AMhzr,\Big\|\frac{2\ell(\overline{b}p_rt_{rA}\overline{r/d}+dk)}{p_rt_{rA}Ah}\Big\|^{-1}\Big)\Big)^\frac{1}{2},
		\end{equation}
		where $\overline{r/d}$ is any multiplicative inverse of $r/d$ modulo $p_rt_{rA}Ah$, and $k=p_rt_{rA}/(r^2z)$ is, by virtue of Lemma~\ref{blemma}, an integer.  Note that, appealing to Cauchy's inequality, we have the majorisation
		\begin{align}\label{new54}
			&\sum_{\substack{h<1/(\xi p_r\eta zd)\\(h,r/d)=1}}\frac{1}{\sqrt{h}}\Big(\sum_{\ell\leqslant(AM)^{1+\varepsilon}hzr}\min\Big(AMhzr,\Big\|\frac{2\ell(\overline{b}p_rt_{rA}\overline{r/d}+dk)}{p_rt_{rA}Ah}\Big\|^{-1}\Big)\Big)^\frac{1}{2}\notag\\
			&\hspace{2mm}\lle(AM)^\varepsilon\Big(\sum_{\substack{h<1/(\xi p_r\eta zd)\\(h,r/d)=1}}\sum_{\ell\leqslant(AM)^{1+\varepsilon}hzr}\min\Big(AMhzr,\Big\|\frac{2\ell(\overline{b}p_rt_{rA}\overline{r/d}+dk)}{p_rt_{rA}Ah}\Big\|^{-1}\Big)\Big)^\frac{1}{2}\notag\\
			&\hspace{2mm}\lle(AM)^{\varepsilon}\Big(\mathop{\sum_{0\leqslant u<A/(\xi p_r\eta zd)}\sum_{h<1/(\xi p_r\eta zd)}\sum_{\ell\leqslant(AM)^{1+\varepsilon}hzr}}_{\substack{\|2\ell(\overline{b}p_rt_{rA}\overline{r/d}+dk)/(p_rt_{rA}Ah)\|=u/(p_rt_{rA}Ah)\\(h,r/d)=1}}\min\Big(AMhzr,\frac{p_rt_{rA}Ah}{u}\Big)\Big)^\frac{1}{2}.
		\end{align}
		For any non-negative integer $u$, we let $\mathfrak{S}(u)$ denote the inner double sum on the right-hand side of \eqref{new54}.\par
		Now, if $h$ and $\ell$ are integers satisfying $\|\ell(\overline{b}p_rt_{rA}\overline{r/d}+dk)/(p_rt_{rA}Ah)\|=0$, then we must at least have $Ah|\ell(\overline{b}p_rt_{rA}+rk)$.  Moreover, since $(A,\overline{b}p_rt_{rA}+rk)=1$, it is clear that $A|\ell$ and therefore $h|(\ell/A)(p_rt_{rA}\overline{b}+rk)$.  This observation yields the estimate
		\begin{equation*}
			\mathfrak{S}(0)\leqslant\sum_{\ell\leqslant2(AM)^\varepsilon Mr/(\xi p_r\eta)}\sum_{\substack{h<1/(\xi p_r\eta zd)\\h|\ell(p_rt_{rA}\overline{b}+rk)}}AMhzr\lle(AM)^{\varepsilon}\frac{A(Mr)^2}{\eta^2},
		\end{equation*}
		which is sufficient in view of \eqref{ggsum} and \eqref{new54}.  To treat the remaining sums $\mathfrak{S}(u)$, we firstly introduce some notation.  We take $X=\min(A/(\xi p_r\eta zd),1/(Mzr))$, and write
		\begin{equation*}
			\nu(U)=\mathop{\sum_{h<1/(\xi p_r\eta zd)}\sum_{u\sim U}\sum_{\ell\leqslant2(AM)^{1+\varepsilon}r/(\xi p_r\eta)}}_{\substack{\|\ell(\overline{b}p_rt_{rA}\overline{r/d}+dk)/(p_rt_{rA}Ah)\|=u/(Ah)\\(Ah,r/d)=1}}1.
		\end{equation*}
		for $U\geqslant1$.  See that $\min(AMhzr,p_rt_{rA}Ah/u)<\min(AMr/(\xi p_r\eta d),A^{1+\varepsilon}/(u\xi p_r\eta zd))$ for any $h<1/(\xi p_r\eta zd)$, and moreover that $AMr/(\xi p_r\eta d)<A/(u\xi p_r\eta zd)$ precisely when $u<X$.  We hence, by virtue of Lemma~\ref{bigprelim}, have the majorisation
		\begin{equation*}
			\sum_{u<X}\mathfrak{S}(u)\lle\frac{(AM)^{1+\varepsilon}r}{\eta}\max_{U<X}\nu(U)\lle\frac{(AM)^{1+\varepsilon}r}{\eta^2z}.
		\end{equation*}
		Similarly, noting that $A/(u\xi p_r\eta zd)<A/(U\eta z)$ whenever $u\sim U$, we have
		\begin{equation*}
			\sum_{X\leqslant u<A/(\xi p_r\eta zd)}\mathfrak{S}(u)\lle\frac{A^{1+\varepsilon}M^\varepsilon}{\eta z}\max_{X\leqslant U<A/(2\xi p_r\eta zd)}\frac{\nu(U)}{U}\lle\frac{(AM)^{1+\varepsilon}r}{\eta^2z},
		\end{equation*}
		on appealing to Lemma~\ref{bigprelim}.  By virtue of \eqref{ggsum} and \eqref{new54}, these estimates are sufficient to establish the desired result.  This completes the proof.
	\end{proof}
	\begin{lemma}\label{rlemma}
		Suppose that the hypothesis of Lemma~\ref{keylemma} holds, and write
		\begin{equation*}
			R(m,\ell)=\Big(\ell-\frac{\xi\eta\ell^2}{r^*AM}\Big)\exp\Big(-\frac{\ell^2}{4(AMzmr^*)^2}\Big),
		\end{equation*}
		as well as
		\begin{equation*}
			\epsilon(m,\ell)=\e\Big(\frac{B\ell}{2Ar^*}+\frac{\ell^2}{4Azm(r^*)^2}-\frac{\overline{4m^*bA}(m^*bB+\ell)^2}{r^*}\Big),
		\end{equation*}
		where $\overline{4m^*bA}$ denotes any multiplicative inverse of $4m^*bA$ modulo $r^*$.  Then, we have
		\begin{equation*}
			\sum_{m<1/(\xi p_r\eta z)}\Big|\sum_{\ell<r^*AM/(\xi\eta)}\frac{R(m,\ell)\epsilon(m,\ell)}{mr^*\sqrt{mr^*}}\Big|\lle(AM)^{1+\varepsilon}\frac{\sqrt{AMz}+Mz\sqrt{Ar}}{\eta},
		\end{equation*}
		where the implied constant depends on $\varepsilon$ alone.
	\end{lemma}
	\begin{proof}
		Purely for the sake of convenience, we split the sum over $\ell$ according to the parity of $\ell$.  Our current objective, therefore, is to estimate the sum
		\begin{equation*}
			\Sigma(m)=\sum_{\ell<r^*AM/(2\xi\eta)}R(m,2\ell)\epsilon(m,2\ell),
		\end{equation*}
		which, according to the Abel summation formula, can be expressed as
		\begin{equation*}
			\Sigma(m)=-\int\limits_0^{r^*AM/(2\xi\eta)}\sum_{\ell\leqslant y}\epsilon(m,2\ell)\:\dif R(m,2y).
		\end{equation*}
		Now, writing $Y(m)=\min(r^*AM/(2\xi\eta),(AM)^{1+\varepsilon}mzr^*)$, we split the above range of integration into the two subranges $\mathscr{Y}_1=\:]0,Y(m)[$ and $\mathscr{Y}_2=\:]Y(m),r^*AM/(2\xi\eta)[$, and call the corresponding integrals $\Sigma_1(m)$ and $\Sigma_2(m)$.  Note that we have defined $Y(m)$ in such a way that $\Sigma_2(m)=O_\varepsilon((AM)^{-8})$.  Moreover, by the fundamental theorem of calculus, we have the majorisation
		\begin{equation}\label{sigma1}
			\Sigma_1(m)\ll Y(m)\max_{y\leqslant Y(m)}\Big|\sum_{\ell\leqslant y}\epsilon(m,2\ell)\Big|.
		\end{equation}
		Now, recall that $\e(\overline{v}/u)\e(\overline{u}/v)=\e(1/(uv))$ holds whenever $u$ and $v$ are coprime, and $\overline{u}$ and $\overline{v}$ are such that $u|(v\overline{v}-1)$ and $v|(u\overline{u}-1)$.  Twice using this fact, we have
		\begin{equation*}
			\sum_{\ell\leqslant y}\epsilon(m,2\ell)=\sum_{\ell\leqslant y}\e\Big(\frac{(\overline{b}p_rt_{rA}\overline{r^*}+k^*)\ell^2}{p_rt_{rA}Am^*}+\frac{\overline{r^*}B\ell}{A}-\frac{\overline{4m^*bA}(m^*bB)^2}{r^*}\Big)\e\Big(-\frac{\overline{b}\ell^2}{Am^*r^*}\Big),
		\end{equation*}
		where $\overline{r^*}$ denotes any multiplicative inverse of $r^*$ modulo $p_rt_{rA}Am^*$.  It thus follows, from the Abel summation formula, that the sum on the right-hand side of \eqref{sigma1} is, for all $y\leqslant Y(m)$, majorised by
		\begin{equation*}
			\Big(1+\frac{Y(m)^2}{Am^*}\Big)\max_{y\leqslant Y(m)}\Big|\sum_{\ell\leqslant y}\e\Big(\frac{(\overline{b}p_rt_{rA}\overline{r^*}+k^*)\ell^2}{p_rt_{rA}Am^*}+\frac{\overline{r^*}B\ell}{A}\Big)\Big|.
		\end{equation*}
		Using the fact that $m^*<1/(\xi p_r\eta z)$, we can show that $Y(m)^2/(Am^*)\ll(AM)^{\varepsilon}$, and thus, appealing to \eqref{sigma1} and Lemma~\ref{weylsum}, we have
		\begin{align}\label{sigma1.1}
			&\Sigma_1(m)\lle(AM)^{1+\varepsilon}mzr^*\Big((AM)^{1+\varepsilon}mzr^*\notag\\
			&\hspace{23mm}+\sum_{\ell\leqslant(AM)^{1+\varepsilon}mzr^*}\min\Big((AM)^{1+\varepsilon}mzr^*,\Big\|\frac{2\ell(\overline{b}p_rt_{rA}\overline{r^*}+k^*)}{p_rt_{rA}Am^*}\Big\|^{-1}\Big)\Big)^\frac{1}{2}.
		\end{align}
		Now, the right-hand side of \eqref{sigma1.1} is indeed a majorant for $\Sigma_2(m)=O_\varepsilon((AM)^{-8})$, and thus it also must be a majorant for $\Sigma(m)$.  Hence, by \eqref{sigma1.1}, we have the bound 
		\begin{equation*}
			\sum_{m<1/(\xi p_r\eta z)}\Big|\sum_{\ell<r^*AM/(2\xi\eta)}\frac{R(m,2\ell)\epsilon(m,2\ell)}{mr^*\sqrt{mr^*}}\Big|\lle(AM)^{1+\varepsilon}\frac{\sqrt{AMz}+Mz\sqrt{Ar}}{\eta},
		\end{equation*}       
		on appealing to Lemma~\ref{prerlemma}.  The desired bound for the subsum over even $\ell$ is thus established.  The same bound can be established for the subsum over odd $\ell$, following the same procedure.  Indeed, we clearly see that the argument of $\epsilon(m,2\ell-1)$ differs from that of $\epsilon(m,2\ell)$ by a polynomial in $\ell$ of degree at most one, which does not affect the application of Lemma~\ref{weylsum} in deriving \eqref{sigma1.1}.  The proof is thus complete.  
	\end{proof}
	\begin{lemma}\label{lastlem}
		Under the hypothesis of Lemma~\ref{keylemma}, we have
		\begin{equation*}
			\frac{\eta z}{AM}\mathop{\sum_{m\in\zz}\sum_{\ell\in\zz}}_{m\ell>0}S(m,\ell)\lle(AM)^\varepsilon\big(\sqrt{M+r}+M\sqrt{zr}\big),
		\end{equation*}
		where the implied constant depends on $\varepsilon$ alone.
	\end{lemma}
	\begin{proof}
		By the symmetry of $\hat{\phi}$, it clearly suffices to bound the subsum over positive $m$ and $\ell$, for the same bound will also hold for complementary subsum.  So, by means of Lemmata~\ref{gausslemma}~and~\ref{integrallemma}, we firstly note that
		\begin{align}\label{old515}
			\sum_{m}\sum_{\ell}S(m,\ell)&\ll\sum_{m<1/(\xi p_r\eta z)}\sum_{\ell<r^*AM/(\xi\eta)}\frac{1}{\sqrt{r^*}}\Big(\frac{1}{zm}+\frac{\sqrt{A}r^*}{\ell\sqrt{zm}}\Big)\notag\\
			&\hspace{22mm}+\frac{1}{z\sqrt{zA}}\sum_{m<1/(\xi p_r\eta z)}\Big|\sum_{\ell<r^*AM/(\xi\eta)}\frac{R(m,\ell)\epsilon(m,\ell)}{mr^*\sqrt{mr^*}}\Big|,
		\end{align}
		where $R(m,\ell)$ and $\epsilon(m,\ell)$ are as in Lemma~\ref{rlemma}.  The second term on the right-hand side of \eqref{old515} can be treated using Lemma~\ref{rlemma}, while for the first term we note that
		\begin{equation*}
			\frac{\eta z}{AM}\sum_{m<1/(\xi p_r\eta z)}\sum_{\ell<r^*AM/(\xi\eta)}\Big(\frac{1}{zm\sqrt{r^*}}+\frac{\sqrt{Ar^*}}{\ell\sqrt{zm}}\Big)\lle(AM)^\varepsilon\sqrt{r}.
		\end{equation*}
		This completes the proof of the assertion.
	\end{proof}
	Using Lemmata~\ref{old63}~and~\ref{lastlem} with Lemma~\ref{keylemma}, we can derive a third estimate for $P_{\delta t_{rA}}(\beta)$, which, together with Lemmata~\ref{prop1}~and~\ref{prop2}, is sufficient to establish our principal estimate.  We shall demonstrate this presently.
	\begin{proof}[Proof of Proposition~\ref{prop}.]
		In the case where $\delta\gg(AM^2)^{-1}$ or $\delta\ll(AM^2)^{-2}$, the assertion follows immediately on appealing to the trivial bound \eqref{trivialbound111}.  Hence, we assume in the remainder of the demonstration that $(AM^2)^{-2}\ll\delta\ll(AM^2)^{-1}$, and note that, in this case, we necessarily have $p_rt_{rA}=O_\varepsilon((AM)^\varepsilon)$.\par    
		Now, we recall that the assertion holds for the second term in the minimum by virtue of Lemma~\ref{prop1} and the fact that $rz\leqslant t_{rA}\sqrt{\delta}$.  Regarding the first term in the minimum, we firstly note that, combining Lemmata~\ref{old63}~and~\ref{lastlem} with Lemma~\ref{keylemma} we obtain an estimate for $P_{\delta t_{rA}}(\beta)$, which, when combined with Lemma~\ref{prop1}, gives us
		\begin{align}\label{newnewnew}
			P_{\delta t_{rA}}(\beta)&\lle(AM)^\varepsilon\big(AM^3\delta+\sqrt{M+r}+M\sqrt{zr}+\min(AM^2zr,M\delta/(z\sqrt{r}))\big)\notag\\
			&\lle(AM)^\varepsilon\big(AM^3\delta+\sqrt{M+r}+M\sqrt[4]{\delta}+M\sqrt{AM\delta}\sqrt[4]{r}\big).
		\end{align}
		Using \eqref{newnewnew} in the case where $r\leqslant M$, and Lemma~\ref{prop2} in the complementary case, we complete the proof on noting that $M\sqrt[4]{\delta}+M\sqrt{AM\delta}\sqrt[4]{r}\ll AM^3\delta+\sqrt{M}$.
	\end{proof}    
	We conclude this section by noting that the second term in the minimum of Proposition~\ref{prop} can be derived by following an argument of Baier \cite{baier2016} involving the square sieve of Heath-Brown \cite{HBsqsieve}.  This argument, however, yields a result slightly weaker than Lemma~\ref{prop1}, and would consequently not be sufficient to derive \eqref{newnewnew} above.  We detail this procedure in the short article \cite{squaresieve}.\newline
	
	\section{Establishment of the main assertions.}
	Having established our principal estimate, Proposition~\ref{prop}, we now have the tools required to prove Theorem~\ref{thm}.  We thus commence the demonstration presently.
	\begin{proof}[Proof of Theorem~\ref{thm}.]  
		By the final assertion of Lemma~\ref{sizeofq}, we see that there exists an absolute constant $\xi\geqslant1$ such that every natural number $q\geqslant\xi$ satisfies $f(q)\sim AM^2$ for some $M\geqslant1$.  Hence, on noting that $\mathscr{N}_f(\xi,N)\leqslant\xi f(\xi)+\xi N$ by virtue of the trivial bound \eqref{trivial}, we see appealing also to the monotonicity of $f$ that $\mathscr{N}_f(Q,N)$ is no greater than
		\begin{equation}\label{dyadic}
			Qf(Q)+\xi N+\frac{\log f(Q)}{\log2}\max_{M\geqslant1}\sup_{(z_n)_n\in\mathscr{C}_N}\sum_{\substack{q\leqslant Q\\f(q)\sim AM^2}}\sum_{\substack{a\leqslant f(q)\\(a,f(q))=1}}\Big|\sum_{n\leqslant N}z_n\e\Big(\frac{an}{f(q)}\Big)\Big|^2.
		\end{equation}
		In order to later apply Proposition~\ref{prop}, we split the maximum appearing in \eqref{dyadic} into two maxima, according to whether $(AM)^2\geqslant|\Delta_f|$ or not.  We firstly treat the latter, wherein, recalling the assumption $Q\geqslant |\Delta_f|/A^2$, we necessarily have $M<\sqrt{Q}$.  Thus, applying the trivial bound \eqref{trivialbound111} with $\delta=N^{-1}$ to \eqref{cls}, we have
		\begin{align}\label{sub2}
			&\max_{M<\sqrt{|\Delta_f|}/A}\sup_{(z_n)_n\in\mathscr{C}_N}\sum_{\substack{q\leqslant Q\\f(q)\sim AM^2}}\sum_{\substack{a\leqslant f(q)\\(a,f(q))=1}}\Big|\sum_{n\leqslant N}z_n\e\Big(\frac{an}{f(q)}\Big)\Big|^2\notag\\
			&\hspace{50mm}\lle (AQ)^\varepsilon\min\big(AQ\sqrt{Q}+N\sqrt{Q},(AQ)^2+N)\notag\\
			&\hspace{50mm}\lle (AQ)^\varepsilon\big(AQ^3+N+\min(N\sqrt{Q},AQ^2\sqrt{N})\big),
		\end{align}
		which is sufficient for our purposes.  Now, to treat the maximum over the remaining range, we simply apply Proposition~\ref{prop} with $\delta=N^{-1}$ to \eqref{cls}, to obtain the bound
		\begin{align}\label{sub3}
			&\max_{M\geqslant\sqrt{|\Delta_f|}/A}\sup_{(z_n)_n\in\mathscr{C}_N}\sum_{\substack{q\leqslant Q\\f(q)\sim AM^2}}\sum_{\substack{a\leqslant f(q)\\(a,f(q))=1}}\Big|\sum_{n\leqslant N}z_n\e\Big(\frac{an}{f(q)}\Big)\Big|^2\notag\\
			&\hspace{50mm}\lle(AQ)^\varepsilon\big(AQ^3+N+\min(N\sqrt{Q},AQ^2\sqrt{N})\big).
		\end{align}
		Combining \eqref{sub2} and \eqref{sub3}, we obtain a majorisation for \eqref{dyadic}, which, in view of the fact from Lemma~\ref{sizeofq} that $f(Q)\asymp AQ^2$, is sufficient for the establishment of the assertion.  The proof is therefore complete.
	\end{proof}
	Now, our demonstration of Theorem~\ref{infthm} is not very involved, and follows almost exactly the work of Baier, Lynch, and Zhao \cite{BaLyZh} from the case of square moduli.  The proof is constructive, and produces a result slightly stronger than the assertion itself.
	\begin{proof}[Proof of Theorem~\ref{infthm}.]
		Suppose that $Q$ is a large real number.  Choosing a specific sequence $(z_n)_n\in\mathscr{C}_N$, depending only on $Q$ and $f(2Q)$, we see that
		\begin{equation}\label{2ndthm1}
			\mathscr{N}_f(2Q,2Qf(2Q))\geqslant\frac{40}{Qf(2Q)}\sum_{q\sim Q}\sum_{\substack{a\leqslant f(q)\\(a,f(q))=1}}\Big|\sum_{n\leqslant Qf(2Q)/40}\e\Big(\frac{an}{f(q)}-\frac{n}{Q}\Big)\Big|^2,
		\end{equation}
		where we have removed the sum over $q\leqslant Q$ for later convenience.  Now, let $Q_m$ denote the product of the first $m$ primes $p$ such that $\Delta_f$ is a non-zero quadratic residue modulo $p$, and write $N_m=2Q_mf(2Q_m)$.  Moreover, let $\mathscr{F}_m$ denote the set of Farey fractions $a/f(q)$ with $q\sim Q_m$, such that $|a/f(q)-1/Q_m|\leqslant10/N_m$.  Then, by virtue of \eqref{2ndthm1}, we see that in order to establish the assertion, it suffices to show that there exists an absolute constant $\xi>0$ such that the bound
		\begin{equation}\label{2ndthm2}
			\sum_{\alpha\in\mathscr{F}_m}\Big|\sum_{n\leqslant N_m/80}\e\Big(\alpha n-\frac{n}{Q_m}\Big)\Big|^2\geqslant\xi N_m^2Q_m^{(\log2)/(2\log\log Q_m)}
		\end{equation}
		holds for all sufficiently large natural numbers $m$.  The remainder of the proof is thus dedicated to the establishment of \eqref{2ndthm2}.\par        
		Now, note that if $q\sim Q_m$ is such that $f(q)\in1+Q_m\zz$, then there must exist a Farey fraction $a/f(q)$ such that $f(q)=1+aQ_m$.  For such $q$, we must therefore have $|a/f(q)-1/Q_m|=1/(Q_mf(q))\leqslant10/N_m$, provided that $m$ is sufficiently large, by virtue of the fact that $5f(Q_m)\geqslant f(2Q_m)$ for all sufficiently large $m$.  Consequently,
		\begin{equation}\label{2ndthm3}
			\#\mathscr{F}_m\geqslant\#\{q\sim Q_m:f(q)\in1+Q_m\zz\}=\prod_{p|Q_m}\#\{q\leqslant p:f(q)\in1+p\zz\},
		\end{equation}
		on appealing to the Chinese remainder theorem.  Now, for every prime $p$ dividing $Q_m$, since $\Delta_f$ is a non-zero quadratic residue modulo $p$, there are exactly two natural numbers $q\leqslant p$ satisfying $f(q)\in1+p\zz$.  Hence, by \eqref{2ndthm3}, we must have
		\begin{equation*}
			\sum_{\alpha\in\mathscr{F}_m}\Big|\sum_{n\leqslant N_m/80}\e\Big(\alpha n-\frac{n}{Q_m}\Big)\Big|^2\geqslant2^m\min_{q\sim Q_m}\Big|\sum_{n\leqslant N_m/80}\e\Big(\frac{n}{Q_mf(q)}\Big)\Big|^2\geqslant\xi2^mN_m^2,
		\end{equation*}
		for some appropriate absolute constant $\xi>0$, where we have used the fact that $0<n/(Q_mf(q))\leqslant\tfrac18$ for all $q\sim Q_m$ and $n\leqslant\tfrac1{80}N_m$.  The bound \eqref{2ndthm2} then follows on noting that, by the prime number theorem, for all sufficiently large $m$, we have
		\begin{equation*}
			\log Q_m=\sum_{p|Q_m}\log p\leqslant2m\log m,\quad\text{so that}\quad m\geqslant\frac{\log Q_m}{2\log\log Q_m}.
		\end{equation*}
		The proof is therefore complete.
	\end{proof}
	From the above demonstration, it is clear that the $f(Q)^\varepsilon$ factor in Theorem~\ref{thm} is due, at least in part, to the moduli $f(q)$ where $q$ has a large number of distinct prime factors.  It would therefore be interesting to see whether this factor could be removed if we instead averaged over primes $p\leqslant Q$ in the definition of $\mathscr{N}_f(Q,N)$.\newline
	
	\section{Demonstration of the corollary.}
	Applications of large sieve inequalities in studying the distribution of zeros for families of Dirichlet $L$-functions have long been known.  We refer the reader to \cite{montgomery,bombierifrench} for a classical introduction to this theory.  The standard procedure, which is based on a method used by Ingham \cite{ing98} to study the zeros of the Riemann $\zeta$-function, consists of two key ingredients.  The first of these ingredients is a mean-value estimate for $L(s,\psi\chi)$, averaged over a finite set of well-spaced points on the critical line.  The second is a mean-value estimate for Dirichlet polynomials, twisted by the character family of interest, and averaged over a finite set of vertically well-spaced points in $R(\sigma,T)$.  As we aim to derive a weighted zero-density estimate, the aforementioned mean-value estimates will need to be correspondingly weighted.\par
	We now fix a natural number $r$, and for brevity suppose that $\mathscr{C}$ is the collection of non-principal Dirichlet characters of the form $\psi\chi$, where $\psi$ is a character modulo $r$ and $\chi$ is a primitive character modulo $f(q)$ for some $q\leqslant Q$ with $(r,f(q))=1$.  Then, appealing to a classical result of Montgomery \cite{montgomery}, we have the majorisation
	\begin{equation}\label{moment}
		\sum_{\psi\chi\in\mathscr{C}}\sum_{t\in\mathscr{T}_{\psi\chi}}\frac{|\tau(\psi)|^2}{\varphi(r)}\big|L\big(\tfrac12+it,\psi\chi\big)\big|^4\lle Q(rf(Q)T)^{1+\varepsilon},
	\end{equation}
	where each $\mathscr{T}_{\psi\chi}\subset[-T,T]$ is such that $|t-t'|\geqslant1$ for all distinct $t,t'\in\mathscr{T}_{\psi\chi}$.  Now, in order to establish our second preliminary result, we must firstly derive a multiplicative analog of Theorem~\ref{thm}.  So, following the procedure outlined in \cite{gallgs}, we note that
	\begin{equation}\label{cor1}
		\sum_{\psi\chi\in\mathscr{C}}\frac{|\tau(\psi\chi)|^2}{\varphi(rf(q))}\Big|\sum_{n\leqslant N}z_n\psi\chi(n)\Big|^2\leqslant\sum_{\substack{q\leqslant Q\\(r,f(q))=1}}\sum_{\substack{a\leqslant rf(q)\\(a,rf(q))=1}}\Big|\sum_{n\leqslant N}z_n\e\Big(\frac{an}{rf(q)}\Big)\Big|^2.
	\end{equation}
	Appealing to some standard results in the theory of Gau\ss\:sums (cf.\:\S3.4 of \cite{iwakow}), we see that $|\tau(\psi\chi)|=|\tau(\psi)|\sqrt{f(q)}$ for all $\psi\chi\in\mathscr{C}$, and consequently
	\begin{equation*}
		\frac{|\tau(\psi\chi)|^2}{\varphi(rf(q))}\geqslant\frac{|\tau(\psi)|^2}{\varphi(r)}.
	\end{equation*}
	Thus, appealing to \eqref{cor1}, we may use Theorem~\ref{thm} to obtain the weighted multiplicative large sieve estimate
	\begin{equation*}
		\sum_{\psi\chi\in\mathscr{C}}\frac{|\tau(\psi)|^2}{\varphi(r)}\Big|\sum_{n\leqslant N}z_n\psi\chi(n)\Big|^2\ll (rf(Q))^\varepsilon\big(rQf(Q)+N\sqrt{Q}\big)\sum_{n\leqslant N}|z_n|^2,
	\end{equation*}
	from which we derive, following the method used to prove Lemma~2.1 of \cite{cozh}, the majorisation
	\begin{align}\label{mvest}
		&\sum_{\psi\chi\in\mathscr{C}}\sum_{s\in\mathscr{S}_{\psi\chi}}\frac{|\tau(\psi)|^2}{\varphi(r)}\Big|\sum_{n\leqslant N}z_n\psi\chi(n)n^{-s}\Big|^2\notag\\
		&\hspace{50mm}\lle(rf(Q)T)^\varepsilon\big(rQf(Q)T+N\sqrt{Q}\big)\sum_{n\leqslant N}|z_n|^2n^{-2\sigma},
	\end{align}
	where each $\mathscr{S}_{\psi\chi}\subset R(\sigma,T)$ is such that $|\Im s-\Im s'|\geqslant1$ for all distinct $s,s'\in\mathscr{S}_{\psi\chi}$.\par
	We now have the preliminaries sufficient to derive Corollary~\ref{cor}, which we shall do shortly.  As the demonstration is fairly standard, however, we shall appeal to Theorems~2.2~and~2.3 of \cite{cozh} in our proof, and refer the reader to \S3.1 of \cite{honoursthesis} for the complete argument.  Indeed, it is not difficult to generalise Theorems~2.2~and~2.3 of \cite{cozh} to include the weights $|\tau(\psi)|^2/\varphi(r)$, and to be applicable more generally to families of non-principal characters.  As these results require an amount of nomenclature to be stated properly, we shall not explicitly write them here in full generality, but instead, when we refer to Theorems~2.2~and~2.3 of \cite{cozh} in the following, it is to be understood that we refer to these slight generalisations of them.
	\begin{proof}[Proof of Corollary~\ref{cor}.]
		Suppose that $1\leqslant\log X\leqslant \log Y\leqslant K\log rQf(Q)T$ for some absolute constant $K\geqslant1$.  Then, for any Dirichlet character $\chi$ write
		\begin{equation*}
			M_X(s,\chi)=\sum_{n\leqslant X}\mu(n)\chi(n)n^{-s}\quad\text{and}\quad H_X(s,\chi)=L(s,\chi)M_X(s,\chi),
		\end{equation*}
		and let $\eta_X(n)$ denote the Dirichlet coefficients of $H_X(s,\iota)$, where $\iota$ is the trivial character modulo 1.  Moreover, for each Dirichlet character $\psi\chi\in\mathscr{C}$, suppose that $\mathscr{R}_{\psi\chi}\subset R(\sigma,T)$ is a subset of the non-trivial zeros of $L(s,\psi\chi)$ such that $|\Im\varrho-\Im\varrho'|\geqslant3$ for all distinct $\varrho,\varrho'\in\mathscr{R}_{\psi\chi}$.  Then, define       
		\begin{equation*}
			R_1(X,Y)=\sup_{X<U<Y^2}\sum_{\psi\chi\in\mathscr{C}}\sum_{\varrho\in\mathscr{R}_{\psi\chi}}\frac{|\tau(\psi)|^2}{\varphi(r)}\Big|\sum_{n\sim U}\eta_X(n)\psi\chi(n)n^{-\varrho}e^{-n/Y}\Big|^2
		\end{equation*}
		and
		\begin{equation*}
			R_2(X,Y)=Y^{2(1-2\sigma)/3}\sum_{\psi\chi\in\mathscr{C}}\sum_{\varrho\in\mathscr{R}_{\psi\chi}}\frac{|\tau(\psi)|^2}{\varphi(r)}\sup_{|t_\varrho-\Im\varrho|\leqslant1}\big|H_X\big(\tfrac12+it_\varrho,\psi\chi\big)\big|^{4/3}.
		\end{equation*}
		By Theorem~2.2 of \cite{cozh}, it is possible to construct the sets $\mathscr{R}_{\psi\chi}$ so that the bound
		\begin{equation}\label{mainrelation}
			\sum_{\psi\chi\in\mathscr{C}}\frac{|\tau(\psi)|^2}{\varphi(r)}N(\sigma,T,\psi\chi)\lle(rf(Q)T)^\varepsilon(R_1(X,Y)+R_2(X,Y))
		\end{equation}
		holds.  On noting that $|\eta_X(n)|\leqslant\tau(n)$ for all $n>X$, we derive from \eqref{mvest} the bound
		\begin{equation}\label{r1bound}
			R_1(X,Y)\lle(rf(Q)T)^\varepsilon\big(rQf(Q)TX^{1-2\sigma}+Y^{2-2\sigma}\sqrt{Q}\big),
		\end{equation}
		and, after applying the inequality of H\"older, we appeal to \eqref{moment} and \eqref{mvest} to obtain the estimate
		\begin{equation}\label{r2bound}
			R_2(X,Y)\lle (rQf(Q)T)^{1/3+\varepsilon}Y^{2(1-2\sigma)/3}\big(rQf(Q)T+X\sqrt{Q}\big)^{2/3}.
		\end{equation}
		In order to establish the assertion with the first term in the minimum, we take
		\begin{equation*}
			X=r\sqrt{Q}f(Q)T\quad\text{and}\quad Y=X^{3/(4-2\sigma)}
		\end{equation*}
		in \eqref{r1bound} and \eqref{r2bound}, and appeal to \eqref{mainrelation}.  In order to establish the assertion with the second term in the minimum, we use \eqref{moment} in conjunction with Theorem~2.3 of \cite{cozh} to obtain the estimate
		\begin{align}\label{cor2}
			&\sum_{\psi\chi\in\mathscr{C}}\frac{|\tau(\psi)|^2}{\varphi(r)}N(\sigma,T,\psi\chi)\notag\\
			&\hspace{23mm}\lle(rf(Q)T)^\varepsilon\big((r^3Qf(Q)^3T^2)^{(1-\sigma)/\sigma}+(r^2f(Q)^2T)^{(1-\sigma)/(2\sigma-1)}\big).
		\end{align}
		The result then follows on noting that the second term on the right-hand side of \eqref{cor2} is smaller than the first term in the range of interest.
	\end{proof}
	We remark here that, under the conjectural bound \eqref{conjecturalbound}, we can derive
	\begin{equation*}
		\sum_{\psi\chi\in\mathscr{C}}\frac{|\tau(\psi)|^2}{\varphi(r)}N(\sigma,T,\psi\chi)\lle(rQf(Q)T)^{3(1-\sigma)/(2-\sigma)+\varepsilon},
	\end{equation*}
	which is an analog of the results of Ingham \cite{ing98} and Montgomery \cite{montgomery}.  Moreover, if we additionally assume the Lindel\"of hypothesis, then we may follow an argument from \S4 of \cite{anote} to establish the density conjecture
	\begin{equation*}
		\sum_{\psi\chi\in\mathscr{C}}\frac{|\tau(\psi)|^2}{\varphi(r)}N(\sigma,T,\psi\chi)\lle(rQf(Q)T)^{2(1-\sigma)+\varepsilon}.
	\end{equation*}\vspace{2mm}
	
	\section*{Acknowledgments.}
	It is a pleasure to record the support of the University of New South Wales and the Australian Government, through the U.\:P.\:A.\:and R.\:T.\:P. Scholarships, respectively.  The author would also like to thank Liangyi~Zhao, whose advice and support throughout the writing of this article is greatly appreciated.  Thanks are additionally due to Stephan Baier, for providing the author with a copy of the published version of the article \cite{baier}, and to Henryk~Iwaniec, for bringing the article \cite{iwaniecsieve} to the attention of the author.\newline

\end{document}